\documentclass[12pt]{amsart}
\usepackage{amsmath,amsfonts,amssymb}
\usepackage{a4wide,amsthm}
\usepackage[english]{babel}
\usepackage[babel=true]{csquotes}
\usepackage{textcomp}
\usepackage{hyperref} 
\usepackage{pdfpages}
\usepackage{pgf,tikz}

\usepackage{mathrsfs}
\usetikzlibrary{arrows}
\usepackage{mathrsfs}
\renewcommand{\div}{\operatorname{div}}
{\newtheorem{thm}{Theorem}[section]}
{\newtheorem{prop}[thm]{Proposition}}
{\newtheorem{coro}[thm]{Corollary}}
{\newtheorem{lemme}[thm]{Lemma}}
{}
{\newtheorem{rem}{Remark}[section]}

\newcommand{\Lip}{\text{Lip}}

\usepackage[symbol]{footmisc}
\author{Amina MECHERBET$^{\dagger}$}
\title{On the sedimentation of a droplet in Stokes flow}
\begin{document}

\footnotetext[2]{Universit\'e de Paris, Institut de Math\'ematiques de Jussieu-Paris Rive Gauche (UMR 7586), F-75205, Paris, France. Email address:  {mecherbet@imj-prg.fr}
}

         \begin{abstract}
              This paper is dedicated to the analysis of a mesoscopic model which describes sedimentation of inertialess suspensions in a viscous flow at mesoscopic scaling. The paper is divided into two parts, the first part concerns the analysis of the transport-Stokes model including a global existence and uniqueness result for $L^1\cap L^\infty$ initial densities with finite first moment. We investigate in particular the case where the initial condition is the characteristic function of the unit ball and show that we recover Hadamard-Rybczynski result, that is, the spherical shape of the droplet is preserved in time.
In the second part of this paper, we derive a surface evolution model in the case where the initial shape of the droplet is axisymmetric. We obtain a 1D hyperbolic equation including non local operators that are linked to the convolution formula with respect to the singular Green function of the Stokes equation. We present a local existence and uniqueness result and show that we recover the Hadamard-Rybczynski result as long as the modelling is well defined and finish with numerical simulations in the spherical case.
          \end{abstract}
          \maketitle
%

\section{Introduction}
In this paper we focus on the problem related to the sedimentation of rigid particles in a fluid where both particles and fluid inertia are neglected. It has been showed in \cite{Hofer,mecherbet,JO,HS} that the limit model for sedimentation when the number of particles tends to infinity while their radius tends to zero is given by the following transport-Stokes model
\begin{equation}\label{Vlasov_Stokes_kappa}
\left\{
\begin{array}{rcll}
\partial_t \rho+{\div}( (u+\kappa g) \rho ) &=& 0\,,& \text{on $\mathbb{R}^+ \times \mathbb{R}^3,$}\\
- \Delta u+ \nabla p &=& 6\pi r_0 \kappa \rho g  \,,& \text{on $\mathbb{R}^+ \times \mathbb{R}^3$},\\
\div u & =&0\,,& \text{on $\mathbb{R}^+ \times \mathbb{R}^3$},\\
\underset{|x|\to \infty}{\lim}|u| &=& 0,& \text{on } \mathbb{R}^+,\\
\rho(0,\cdot) &= & \rho_0 \,,&\text{on $ \mathbb{R}^3$}.
\end{array}
\right.
\end{equation}
Here, $\rho$ stands for the probability density function of the particles, $(u,p)$ are the velocity and pressure of the fluid, $g$ is the gravity vector,  $R=\frac{r_0}{N}$ is the radius of the particles where $N$ the (large) number of particles in the cloud and $\kappa g= \frac{2}{9} {R^2}(\bar{\rho}-\rho)g$ represents the fall speed of one particle sedimenting under gravitational force. Note in particular that the source term in the Stokes equation corresponds to $6\pi r_0\kappa g \rho = N \frac{4}{3}\pi R^3(\rho_p-\rho_f) g \rho = \phi  (\rho_p-\rho_f)g\rho $ where $\phi$ is the solid volume fraction of the suspension in the case $ |\text{supp } \rho| =1$.\\

A first quantitative estimate for the convergence in a mean-field setting using the first-Wasserstein metrics has been proved in \cite{mecherbet} under assumptions on the minimal distance between particles and for $r_0=RN$ small enough. Precisely, for all $T>0$ it is shown that there exists $N_0\in \mathbb{N}$ and two constants $C_1,C_2$ such that for all $N\geq N_0$ and $t\leq T$ we have
\begin{equation}\label{convergence}
W_1(\rho^N(t,\cdot),\rho(t,\cdot)) \leq C_1(o_N(1)+ W_1(\rho^N_0,\rho_0))e^{C_2t},
\end{equation}
where $\rho^N(t,\cdot)= \frac{1}{N} \underset{i}{\sum} \delta_{x_i(t)}$ is the empirical measure describing the cloud of particles at microscopic scaling having positions $(x_i(t))_{1\leq i \leq N}$ at time $t$. We refer also to \cite{HS} for a more recent mean-field analysis including the effective viscosity approximation in the modelling under different assumptions on the separation between the particles. In particular, \eqref{convergence} shows that the limit ``dispersed droplet" described by $\rho$ and the discrete cloud of particles described by $\rho^N$ have a similar behaviour on a finite time interval if they are close enough initially. We emphasize that \eqref{Vlasov_Stokes_kappa} can be seen as an inertialess version of the Vlasov Navier-Stokes equation which are used in practice to describe sprays in fluid-kinetic theory where the volume fraction of the suspension is small but not negligible, see
\cite{BDGM,BGLM15,Moussa} and the references therein.

The problem related to the shape evolution of a falling suspension drop in a viscous fluid has attracted a lot of attention. Numerical simulations and physical experiments in laboratory were carried out; the numerical simulations consists in solving numerically the trajectories of a large number of particles whereas the physical experiments consists in injecting a viscous liquid inside a lighter viscous liquid and track the shape evolution of the falling droplet. In \cite{M&al,NB,BKHM} authors claim that the behaviour of the falling drop depends on its initial shape and that an initially spherical cloud
retains a roughly spherical shape while settling at low Reynolds number whereas they observe a torus formation if the initial shape deviates from the spherical one. 
Precisely, the particles at top of
the cloud leak away from the cluster and form a vertical tail. The decrease of the number of particles at the vertical axis of the cloud leads to the apparition of the toroidal form. In the paper \cite{MNG} authors emphasize that the torus instability occurs also in the case where the initial shape of the blob is spherical and explain that it is a slow process and is likely to happen for a large number of particles which explains why it has not been detected in the former papers.
It is also important to emphasize that it is possible to represent the cloud as an effective medium of excess
mass and the flow system related to that of the sedimentation of a spherical drop of heavy fluid in an otherwise lighter fluid solved by Hadamard \cite{Hadamard} and Rybczynski \cite{rybczynski} in 1911. This macroscopic representation corresponds to a coupled Stokes-Stokes model describing sedimentation of a viscous spherical drop in a viscous fluid assuming a uniform surface tension on the sphere. Using the Stokes stream function for axisymmetric flow, authors show that the spherical shape of the drop is preserved. We refer to \cite{MNG} for more details on the comparison with the numerical simulations and the departure of the trajectories of the particles from Hadamard–Rybczynski streamlines. In particular, in \cite{MNG}, authors explain that the fundamental difference of behaviour while solving numerically a large system for the particle trajectories is due to fluctuations in particle velocity.

In this paper we are interested in providing an analysis for the transport-Stokes model. First we present a global existence and uniqueness result for the transport-Stokes equation for $L^1\cap L^\infty$ initial densities with finite first moment allowing to tackle the case where the density is the characteristic function of a smooth bounded domain. We show in particular that we recover the Hadamard and Rybczynski result in the case where the initial domain is spherical. We emphasize that, to our knowledge, this is the first time that a comparison between the Hadamard-Rybczynski result and the transport-Stokes model is investigated.

The second motivation is to derive a model describing the surface evolution of a falling droplet and provide an analysis of the solution.
The obtained model is a hyperbolic equation describing the evolution of the surface of axisymmetric drop $B_0$. The advantage of such a model for the surface evolution is that the equations are reduced to a 1D problem on a bounded interval which reduces in particular the numerical issues in comparison to the Stokes problem on $\mathbb{R}^3$. We present a local existence and uniqueness result for the hyperbolic model as well as a comparison to the Hadamard-Rybczynski result and finish by providing numerical simulations in this case.

\section{Statement of the main results}
The main results of this paper are divided into two parts that we make precise in this section. The first part deals with the analysis of the transport-Stokes model while the second part deals with the derivation of a surface evolution model and its analysis. As a consequence, the reader only interested in the analysis of the transport-Stokes model and its comparison to the Hadamard-Rybczynski result may focus on section 3.
\subsection{Analysis of the transport-Stokes model}
\subsubsection{Global existence and uniqueness result}
 Existence and uniqueness of \eqref{Vlasov_Stokes_kappa} has been proved in \cite{Hofer} for regular initial data $\rho_0$. The first step of this study is to extend the result for less regular data allowing to tackle blob distribution. 
Note that, as explained in \cite{Hofer}, if $(\rho,u)$ are solutions to equation \eqref{Vlasov_Stokes_kappa}, then $$(\tilde{\rho}(t,x),\tilde{u}(t,x))=(\rho(t,x+t \kappa g),u(t,x+t\kappa g)),$$  is solution to
\begin{equation*}
\left\{
\begin{array}{rcll}
\partial_t \rho+{\div}( \rho u) &=& 0\,,& \text{on $\mathbb{R}^+ \times \mathbb{R}^3,$}\\
- \Delta u+ \nabla p &=&6\pi r_0 \kappa \rho g\,,& \text{on $\mathbb{R}^+ \times \mathbb{R}^3,$}\\
\div u & =&0\,,& \text{on $\mathbb{R}^+ \times \mathbb{R}^3,$}\\
\rho(0,\cdot) &= & \rho_0\,,& \text{on $\mathbb{R}^3.$}
\end{array}
\right.
\end{equation*}
Since $6\pi r_0 \kappa g = -6\pi r_0 \kappa |g| e_3$, without loss of generality, we consider in this paper the following transport-Stokes problem:
\begin{equation}\label{Vlasov_Stokes}
\left\{
\begin{array}{rcll}
\partial_t \rho+{\div}(  \rho u) &=& 0\,,& \text{on $\mathbb{R}^+ \times \mathbb{R}^3,$}\\
- \Delta u+ \nabla p &=& -\rho e_3\,,& \text{on $\mathbb{R}^+ \times \mathbb{R}^3,$}\\
\div u & =&0\,,& \text{on $\mathbb{R}^+ \times \mathbb{R}^3,$}\\
\rho(0,\cdot) &= & \rho_0\,,& \text{on $\mathbb{R}^3.$}
\end{array}
\right.
\end{equation}
where $e_3$ is the third vector of the standard basis in $\mathbb{R}^3$.\\
The first result is a proof of existence and uniqueness of solutions for the transport-Stokes problem.
\begin{thm}\label{thm}
Let $\rho_0 \in L^1(\mathbb{R}^3) \cap L^\infty(\mathbb{R}^3)$ a measure with finite first moment. There exits a unique couple $(\rho,u) \in L^\infty(0,T; L^1(\mathbb{R}^3) \cap L^\infty(\mathbb{R}^3))\times  L^\infty(0,T; W^{1,\infty}(\mathbb{R}^3))$ satisfying the transport-Stokes equation 
\eqref{Vlasov_Stokes} for all $T\geq 0$. Moreover, 
for all $s\in[0,T]$ there exists a unique characteristic flow $X(\cdot,s,\cdot)\in L^\infty(0,T,W^{1,\infty}(\mathbb{R}^3))$ 
\begin{equation*}
\left \{
\begin{array}{rcll}
\partial_t X(t,s,x) & = & u(s, X(t,s,x)),&\: \: \forall \, t,s\in[0,T] , \\
X(s,s,x) & = & x,& \: \: \forall \, s \in [0,T] ,
\end{array}
\right.
\end{equation*}
For all $s,t \in [0,T]$ the diffeomorphism $X(s,t,\cdot)$ is measure preserving and we have
$$
\rho(t,\cdot)=X(t,0,\cdot) \# \rho_0.
$$
\end{thm}
\begin{rem}
A similar well-posedness result has been shown recently in \cite{HS}.
\end{rem}
\begin{rem}
Since $\rho \in L^\infty( 0,T; L^p(\mathbb{R}^3))$ for all $p\in[1,+\infty] $ this ensures in particular that $u\in L^\infty(0,T, W^{2,p}(\mathbb{R}^3))$ for all $p>3$, see \cite[Theorem IV.2.1]{Galdi}, which yields $u \in L^\infty( 0,T;\mathcal{C}^{1,\mu}(\mathbb{R}^3))$ for any $0<\mu<1$.
\end{rem}

This result ensures the well-posedness of the transport-Stokes equation when the initial density is the characteristic function of a smooth bounded domain $B_0\subset \mathbb{R}^3$. In particular we have 
\begin{coro}\label{coro_regularite_surface}
Let $\rho_0=1_{B_0}$ with $B_0$ a smooth bounded domain such that $\Gamma_0:= \partial B_0 \in \mathcal{C}^{0,1}$ (resp. $\partial B_0 \in \mathcal{C}^{1,\mu}$ for some $0<\mu<1$). Then the surface regularity is propagated in time for all $T>0$ : for all $t \geq 0$, $\Gamma_t:=\partial B_t= X(t,0,\partial B_0)  \in \mathcal{C}^{0,1}$ (resp. $\partial B_0 \in \mathcal{C}^{1,\mu}$ for some $0<\mu<1$).  
\end{coro}
\begin{rem}
Corollary \ref{coro_regularite_surface} excludes the formation of a torus, in finite time, if initially $\partial B_0$ is homeomorphic to a sphere.
\end{rem}
\subsubsection{Comparison to the Hadamard and Rybczynski result}

In 1911, Hadamard \cite{Hadamard} and Rybczynski \cite{rybczynski} investigated independently the motion of a liquid spherical drop $B$ falling in a viscous fluid, see also also \cite{Batchelor_book,CGW,Feuillebois}.  
The equations considered are Stokes equations on both fluid and drop domain. Denoting by $\bar{\rho}$, $\bar{\mu}$ (resp.${\rho}$, ${\mu}$) the density and viscosity of the drop (resp. density and viscosity of the fluid), authors show that if $B_0$ is the unit ball then the spherical form of the droplet is preserved and
the velocity fall of the droplet $v^*$ is given by 
\begin{equation}\label{V^*}
v^*= \frac{2}{9} \frac{R^2}{\mu}(\bar{\rho}-\rho) \frac{\mu + \bar{\mu}}{\bar\mu+ \frac{2}{3} {\mu}} g .
\end{equation}
In the case where we drop the coefficient $\frac{R^2}{\mu}(\bar{\rho}-\rho)$ and set $\mu=\bar \mu=1$, we get 
\begin{equation}\label{formule_v^*}
v^*=-\frac{4}{15}e_3.
\end{equation}
We have the following result which shows a similar result for the transport-Stokes system
\begin{prop}\label{conservation_sphere}
The solution $(u,\rho)$ of the transport-Stokes equation \eqref{Vlasov_Stokes} in the case where $\rho_0=1_{B_0}$, $B_0=B(0,1)$ is given by
$$u(t,x)=u_0(x-v^*t),$$
$$\rho(t,x)=\rho_0(x-v^*t), $$
where $ u_0$ solves $-\Delta u +\nabla p = -\rho_0e_3 $, $\div(u_0)=0 $ on $\mathbb{R}^3$ and $v^*$ given by \eqref{formule_v^*}. In particular this shows that $B_t=B_0+v^*t$ for all $t\geq0$.
\end{prop}

The proof relies on a property proven by Hadamard-Rybczynski which states that the normal component of the velocity of the fluid is constant on the surface of the sphere. This constant velocity corresponds to $v^*$ and in particular we present direct computations showing the Hadamard-Rybczynski property, see Lemma \ref{lemme_hadamard}.
\begin{rem}
The propagation of the spherical shape for the transport-Stokes model seems in contradiction with the numerical observations provided in \cite{MNG} but it is worth emphasizing that \eqref{convergence} ensures the similarity of behaviour between the transport-Stokes model and the microscopic one under the assumption that the limit density $\rho$ lies in $L^\infty\cap L^1$ and for a finite time interval $(0,T)$ whereas the departure from the Hadamard and Rybczynski phenomena observed in \cite{MNG} requests a ``very long simulation", see \cite[Section 9]{MNG}. On the other hand, the numerical fluctuations while solving the large system of trajectories may also be investigated in order to justify the fundamental difference of behaviour.
\end{rem}
\subsection{Derivation of a surface evolution model} 
The second part of this paper is devoted to the derivation of a hyperbolic equation describing the surface evolution of a class of initial shapes $B_0$ when $\rho_0=1_{B_0}$. 
 We consider the case of initial axisymetric domains $B_0$ (invariant under rotations around the vertical axis $e_3$) described using a spherical parametrization and a radius function $r_0$ depending only on $\theta \in [0,\pi]$
\begin{equation}\label{def_B_0}
B_0=\left\{r_0(\theta) \begin{pmatrix}
\cos(\phi) \sin(\theta)\\
\sin(\phi) \sin(\theta)\\
\cos(\theta)
\end{pmatrix}, (\theta,\phi)\in[0,\pi]\times[0,2\pi]\right\}.
\end{equation}
The motivation of considering such domains is that the Stokes equation preserves the invariance and ensures that $B_t$ is axisymetric. We set then $c(t)=(0,0,c_3(t))\in B_t$ the position at time $t$ of a reference point such that $c(0)=0$ and write $B_t= c(t)+ \tilde{B}_t$ where
\begin{equation}\label{def_B_t}
\tilde{B}_t=\left\{r(t,\theta) \begin{pmatrix}
\cos(\phi) \sin(\theta)\\
\sin(\phi) \sin(\theta)\\
\cos(\theta)
\end{pmatrix}, (\theta,\phi)\in[0,\pi]\times[0,2\pi]\right\}.
\end{equation}
\begin{rem}
The reference center $c$ is not necessarily the center of mass of the droplet $B_t$. The decomposition $B_t=c(t)+\tilde{B}_t$ with $\tilde{B}_t$ defined in \eqref{def_B_t} is valid as long as $c(t)\in B_t$.
\end{rem}

Using the weak formulation of the transport-Stokes equation we derive a hyperbolic equation for the evolution of the radius $r$.
\begin{equation}\label{equation_r_transport}
\left \{
\begin{array}{rcl}
\partial_t r + \partial_\theta r A_1[r] &= &A_2[r], \\
r(0,\cdot)&=& r_0.
\end{array}
\right.
\end{equation}
The operators $A_1$ and $A_2$ are defined in \eqref{eqA1} and \eqref{eqA2} in Proposition \ref{prop_model}. See also Appendix \ref{appendiceA} for a summary of the formulas. These operators depend non linearly and non locally on the unknown $r$, they also depend on the reference center $c$. We emphasize that there is a coupling between the evolution of the radius $r$ and the motion of the reference center $c$. Precisely, the velocity of $c$ can be seen as a parameter in the model. In particular, if we choose $c$ to be transported along the flow we get $c=c[r]=(0,0,c[r]_3)$ with
\begin{equation}\label{dotc_3}
\left\{
\begin{array}{rcl}
\dot{c}[r]_3(t)&=& -\frac{1}{4}\displaystyle{\int_0^\pi} r^2(t,\bar\theta) \sin(\bar \theta) \left(1-\frac{1}{2} \sin^2(\bar \theta)\right) d \bar \theta ,\\
c[r]_3(0)&=&0,
\end{array}
\right.
\end{equation}
see Proposition \ref{prop_model}.
\subsubsection{Local existence and uniqueness result for the surface evolution equation}
We present a local existence and uniqueness result of $(r,c)$ for Lipschitz functions $r_0$ such that
$$
|r|_*=\underset{(0,\pi)}{\inf}r(\theta)>0.
$$
We emphasize that this existence and uniqueness result cannot be obtained as a direct consequence of the global existence and uniqueness result for the transport-Stokes equation since the validity of the spherical parametrization is not ensured a priori. This motivates the following Theorem 
\begin{thm}\label{thm2}
Let $r_0\in\mathcal{C}^{0,1}[0,\pi]$ such that $|r_0|_*>0$. There exists $T>0$ and a unique $r\in \mathcal{C}(0,T;\mathcal{C}^{0,1}(0,\pi))$ satisfying the hyperbolic equation \eqref{equation_r_transport}. Moreover, there exists a unique associated reference point $c=c[r]\in \mathcal{C}(0,T)$ satisfying \eqref{dotc_3} .
\end{thm}
\begin{rem}
The same result holds true if the motion of the center $c$ is defined in another way. The only properties needed is a uniform bound on $\dot{c}$ and a stability estimate with respect to $r$ if $c=c[r]$, see \eqref{estimate_c}. 
\end{rem}

\subsubsection{Investigation of the spherical case and numerical simulations}
In Section \ref{section_hyperbolic_sphere} we investigate the spherical case starting from the hyperbolic equation. We emphasize that this cannot be obtained directly from Proposition \ref{coro_regularite_surface} since the hyperbolic model depends on the choice of the reference center $c$. More precisely we distinguish two cases: if ${c}={c}^*$, then a straightforward computation shows that $A_2[1]=0$ in \eqref{equation_r_transport} and hence $r=1$ is solution of the hyperbolic equation and we recover the Hadamard-Rybczynski result. On the other hand, if $\dot{c}\neq\dot{c}^*$, we show that the solution $r$ corresponds to a spherical parametrization of the Hadamard-Rybczynski sphere $B(c^*,1)$ as long as the reference center $c$ belongs to $B(c^*,1)$, see Proposition \ref{prop_alkashi}. Moreover, in the case where $\dot{c}$ is given by \eqref{dotc_3}, explicit computations show that $|c(t)-c^*(t)|\leq 1$ for all $t\geq 0$ and $\underset{t \to \infty}{\lim} |c(t)-c^*(t)|=1$, see Proposition \ref{prop_cc*}. This ensures that $c(t) \in B(c^*,1)$ for all time and shows global existence of the solution of equations \eqref{equation_r_transport}, \eqref{dotc_3}. This result suggests that the maximal time of existence of the solution $r$ may depend on the choice of the velocity of the reference center $c$. 

We finish the paper by proposing a numerical scheme in order to illustrate the spherical case $r_0=1$. First, we present numerical simulations for the solution of \eqref{equation_r_transport} with $\dot{c}$ fixed as in \eqref{dotc_3} and recover numerically Hadamard-Rybczynski solution. Second we investigate a test case for which $\dot{c}\neq \dot{c}^* $ and is such that the center $c(t)$ leaves the sphere $B(c^*,1)$ after $t=0.5$. Numerical computations show the validity of Proposition \eqref{prop_alkashi} until $t=0.5$ and we observe negative values of the radius $r$ after $t=0.5$. The last test case illustrates the steady state \textit{i.e.} $c=c^*$ for which $r=1$ is solution for all time. We finish by a discussion on the approximation scheme and possible future numerical investigations.
 
\section{On the transport-Stokes model}
In order to prove Theorem \ref{thm} we recall first some existence, uniqueness and stability estimates for Stokes and transport equations.
\subsection{Reminder on the Steady Stokes and transport equations}
Equation \eqref{Vlasov_Stokes} is a steady Stokes problem coupled with a transport equation. We recall here some properties concerning the Stokes problem on $\mathbb{R}^3$ and the transport equations.
\begin{prop}\label{prop1}
Let $\eta \in L^\infty(\mathbb{R}^3) \cap L^1(\mathbb{R}^3)$, The unique velocity field $u$ solution to the Stokes equation: 
\begin{equation*}
\left\{
\begin{array}{rcll}
-\Delta u + \nabla p &=& \eta \,,&\text{ on $\mathbb{R}^3$}\\
\div (u) &=& 0\,,&\text{ on $\mathbb{R}^3$,}
\end{array}
\right.
\end{equation*}
is given by the convolution of the source term $\eta$ with the Oseen tensor $\Phi$
\begin{equation}\label{c1:Oseen}
\Phi(x)= \frac{1}{8 \pi} \left(\frac{\mathbb{I}_3}{|x|}+ \frac{x \otimes x}{|x|^3} \right).
\end{equation}
Moreover, $u\in W^{1,\infty}(\mathbb{R}^3)$ and there exists a positive constant independent of the data such that:
\begin{equation}\label{stab_prop1}
\|u \|_\infty + \|\nabla u \|_\infty \leq C \|\eta\|_{L^1 \cap L^\infty}.
\end{equation}
\end{prop}
A proof can be found in \cite[Lemma 3.18]{Hofer} in the case $\eta \in X_\beta$ where $X_\beta$ is defined in \cite[Definition 2.5]{Hofer}. The proof is mainly the same when considering $\eta \in L^1 \cap L^\infty$. We recall now a stability estimate using the first Wasserstein distance $W_1$ which is well defined for measures with finite first moment. The following Proposition uses arguments similar to \cite[Proposition 3]{HJ} and \cite[Theorem 3.1]{Hauray}.  We refer the reader to \cite{Santambrogio, Villani} for more details on the Wasserstein metrics.
\begin{prop}[Steady-Stokes stability estimates]\label{stab_Stokes}
Let $\eta_1$, $\eta_2 \in L^1(\mathbb{R}^3) \cap L^\infty(\mathbb{R}^3)$ and denote by $u_1$ and $u_2$ the associated Stokes solution. For all compact subset $K \subset \mathbb{R}^3$ one can show that there exists a constant depending on $K$ such that
$$\|u_1 -u_2\|_{L^1(K)}+\| \nabla u_1 -\nabla u_2\|_{L^1(K)} \leq  C(K)\,  W_1(\eta_1,\eta_2).$$
Moreover, given a density $\rho \in L^1 \cap L^\infty$, there exists a positive constant independent of the data such that: 
\begin{equation}\label{formule2}
\int_{\mathbb{R}^3} |u_1(x)-u_2(x)| \rho(dx) \leq C \|\rho\|_{L^1 \cap L^\infty} W_1(\eta_1,\eta_2)
\end{equation}
\end{prop}
Since similar computations will be used thereafter, we present the proof of the former Proposition.
\begin{proof}
According to \cite[Theorem 1.5]{Santambrogio}, there exists an optimal transport map $T$ such that $\eta_2:= T \# \eta_1$ and we have:
$$
W_1(\eta_1,\eta_2) = \int_{\mathbb{R}^3}|T(y)-y|\, \eta_1(dy).
$$
This yields:
\begin{align*}
\int_K \left |u_2(x)-u_1(x)\right| dx &= \int_K \left| \int_{\mathbb{R}^3} \Phi(x-y) \eta_1(dy)-\int_{\mathbb{R}^3} \Phi(x-T(y)) \eta_1(dy) \right| dx \\
&\leq C \int_K \int _{\mathbb{R}^3} \frac{|T(y)-y|}{\min(|x-y|^2,|x-T(y)|^2)}\eta_1(dy) dx \\
& \leq \int_{\mathbb{R}^3} \int_K  \left ( \frac{1}{|x-y|^2}+\frac{1}{|x-T(y)|^2}\right)dx  |T(y)-y|\, \eta_1(dy)\\
& \leq C(K)  W_1(\eta_1,\eta_2).
\end{align*}
The proof of the last formula \eqref{formule2} is analogous to the estimate above where we replace $C(K)$ by $C(\|\rho\|_{L^1 \cap L^\infty})$.
\end{proof}

Given a velocity field having the same regularity as above, we recall now an existence, uniqueness and stability estimate for the transport equations using the first Wasserstein distance.
\begin{prop}\label{stab_transport}
Let $u \in L^\infty(0,T; W^{1,\infty}(\mathbb{R}^3))$ and $\rho_0 \in L^1 \cap L^\infty$, for all $T>0$ there exists a unique solution $\eta\in L^\infty(0,T;L^1 \cap L^\infty)$ to the transport equation
\begin{equation}
\left\{
\begin{array}{rcl}
\partial_t \rho + \div(\rho u )& =& 0 \,, \\
\rho(0,\cdot)&=& \rho_0\,.
\end{array}
\right.
\end{equation}
Moreover, given two velocity fields $u_i$, $i=1,2$, if we denote by $\rho_i$ the solution to the associated transport equation , we have for all $t\geq s \geq 0$:
\begin{multline}\label{formule_stab}
W_1(\rho_1(t),\rho_2(t))\\ \leq \left ( W_1(\rho_1(s),\rho_2(s))+ \int_s^t\int_{\mathbb{R}^3} \left |u_2(\tau,x)-u_1(\tau, x)\right| \rho_1(\tau,x) dx d\tau \right) e^{Q_2(t-s)} ,
\end{multline}
where $Q_i:= \| u_i\|_{L^\infty(0,T; W^{1,\infty})}$.
\end{prop}
\begin{proof}
Classical transport theory ensures the existence and uniqueness. Precisely, the characteristic flow satisfying
\begin{equation}\label{flow_X}
\left \{
\begin{array}{rcll}
\partial_t X(t,s,x) & = & u(s, X(t,s,x)),&\: \: \forall \, t,s\in[0,T] , \\
X(s,s,x) & = & x,& \: \: \forall \, s \in [0,T] ,
\end{array}
\right.
\end{equation}
is well defined in the sense of Carath\'eodory since $u$ is $L^\infty$ in time and Lipschitz regarding the space variable. Moreover, the following formula hods true
\begin{equation}\label{formule1}
\rho(t,\cdot)=X(t,s,\cdot)\# \rho(s,\cdot).
\end{equation}
Now, consider two velocity fields $u_i \in L^\infty(0,T;W^{1,\infty})$ and denote by $X_i$ its associated characteristic flow.
For all $x \neq y $, $i=1,2$ we have:
\begin{align*}
|X_i(t,s,x)-X_i(t,s,y)|& \leq |x-y| + \int_s^t |u_i(\tau,X_i(\tau,s,x))-u_i(\tau,X_i(\tau,s,y))| d \tau \\
& \leq |x-y| + Q_i \int_s^t |X_i(\tau,s,x)-X_i(\tau,s,y)|d\tau\,,
\end{align*}
which yields, using Gronwall's inequality, for all $t \geq s \geq 0$:
$$\Lip(X_i(s,t,\cdot)) \leq e^{Q_i(t-s)}.$$
We recall that at time $s\geq 0$, according to \cite[Theorem 1.5]{Santambrogio}, one can choose an optimal mapping $T_s$ such that $\rho_2(s)= T_s \# \rho_1(s)$ and
$$
W_1(\rho_1(s),\rho_2(s)) := \int |T_s(y) - y| \rho_1(s,dy),
$$
on the other hand, thanks to the flows $X_i$ we can construct a mapping $T_t$ at time $t\geq s$ such that $\rho_2(t)= T_t \# \rho_1(t)$ defined by 
\begin{equation}\label{formule1bis}
T_t:= X_2(t,s,\cdot) \circ T_s \circ X_1(s,t,\cdot). 
\end{equation}
According to the definition of the Wasserstein distance and formulas \eqref{formule1}, \eqref{formule1bis} we have:
\begin{align*}
W_1(\rho_1(t),\rho_2(t))& \leq   \int \left |T_t(x)- x)\right| \rho_1(t,dx) \\
& = \int \left |T_t(X_1(t,s,y))- X_1(t,s,y)\right| \rho_1(s,dy)\\
& = \int \left |X_2(t,s,T_s(y))- X_1(t,s,y)\right| \rho_1(s,dy)\\
& \leq \Lip(X_2(t,s,\cdot))  W_1(\rho_1(s),\rho_2(s))+\int \left |X_2(t,s,y)- X_1(t,s,y)\right| \rho_1(s,dy).
\end{align*}
Now we have:
\begin{align*}
& \int  \left |X_2(t,s,y)- X_1(t,s,y)\right| \rho_1(s,dy)\\
& \leq \int_s^t \int \left |u_2(\tau,X_2(\tau,s,y))-u_1(\tau, X_1(\tau,s,y)\right| \rho_1(s,dy)  d\tau\\
&\leq Q_2 \int_s^t \int  \left |X_2(\tau,s,y)- X_1(\tau,s,y))\right| \rho_1(s,dy)  d\tau + \int_s^t\int \left |u_2(\tau,x)-u_1(\tau, x)\right| \rho_1(\tau,dx) d\tau.
\end{align*}
Gronwall's inequality yields: 
\begin{equation*}
\int_{\mathbb{R}^3}  \left |X_2(t,s,y)- X_1(t,s,y)\right| \rho_1(s,dy)  \leq \left( \int_s^t\int_{\mathbb{R}^3} \left |u_2(\tau,x)-u_1(\tau ,x)\right| \rho_1(\tau,dx) d\tau \right) e^{Q_2(t-s)}.
\end{equation*}
Finally we get
\begin{align*}
W_1(\rho_1(t),\rho_2(t))& \leq \Lip(X_2(t,s,\cdot))  W_1(\rho_1(s),\rho_2(s))\\
&+\left( \int_s^t\int_{\mathbb{R}^3} \left |u_2(\tau,x)-u_1(\tau,x)\right| \rho_1(\tau,dx) d\tau \right) e^{Q_2(t-s)},
\end{align*}
with $\Lip(X_2(s,t,\cdot)) \leq e^{Q_2(t-s)}$.
\end{proof}
\subsection{proof of the existence and uniqueness result}
One can use a fixed-point argument in order to show existence and uniqueness however it is also possible to take advantage of the existence and uniqueness result provided by H\"ofer in \cite{Hofer} for initial regular data. We recall first the definition of the space $X_\beta$ for $\beta\geq 0$ 
$$X_\beta= \{h\in L^\infty(\mathbb{R}^3), \|h\|_{X^\beta}= \underset{x}{\sup} (1+|x|^\beta)|h(x)| < +\infty \}.$$
\begin{thm}[R.~M. H\"ofer 2018]\label{thm_Hofer}
Assume $\rho_0 \in X_\beta $ with $\nabla \rho_0 \in X_\beta $ for some $\beta >2$. Then equation \eqref{Vlasov_Stokes} admits a unique solution $(\rho,u) \in W^{1,\infty}(0,T; X_\beta)\times L^\infty(0,T; W^{1,\infty}(\mathbb{R}^3)) $ for all $T>0$ and $\nabla \rho \in L^\infty(0,T; X_\beta)$.
\end{thm}  

The idea of proof is then to regularize the initial data and use completeness and compactness arguments.
\begin{proof}[Proof of Theorem \ref{thm}]  
Let $\rho_0 \in L^\infty \cap L^1$ a measure with finite first moment. We introduce a mollifier $\chi \geq 0 $ such that $\int \chi=1$ and set $\rho^n_0=\rho_0\star \chi_n$ with $\chi_n(x)=n^3\chi(n x)$. It is then clear that $\rho_0^n $ has a finite first moment and
\begin{eqnarray*}
W_1(\rho_0,\rho_0^n) \underset{ n \to \infty}{\to} 0,& \|\rho_0^n \|_1 \leq \|\rho_0\|_1, &  \|\rho_0^n \|_\infty \leq C_\chi \|\rho_0\|_\infty.
\end{eqnarray*}
Since $\rho_0^n, \nabla \rho_0^n \in X_\beta$, according to Theorem \ref{thm_Hofer}, there exists $(\rho^n,u^n)$ the unique solution to the transport-Stokes equation \eqref{Vlasov_Stokes} for all $T>0$. On the other hand, stability estimates \eqref{formule_stab} and \eqref{formule2} yield for all $m>n\geq 0$
\begin{align*}
W_1(\rho^n,\rho^m) \leq W_1(\rho_0^n,\rho_0^m) + C e^{Q_m t} \|\rho^n\|_{L^1 \cap L^\infty} \int_0^t W_1(\rho^n(\tau), \rho^{m}(\tau))d \tau,
\end{align*}
with $$Q_m :=\underset{\tau \leq t}{\sup} \,  \Lip (u^{m}(\tau,\cdot))\leq \underset{\tau \leq t}{\sup}\, \|u^{m}(\tau,\cdot)\|_{W^{1,\infty}} $$
Since $\rho^n$ is transported by an incompressible fluid we have for all time $t\in [0,T]$:
$$
\|\rho^n (t)\|_{L^1 \cap L^\infty} \leq\|\rho_0^n\|_{L^1 \cap L^\infty} \leq C,
$$
and formula \eqref{stab_prop1} from Proposition \ref{prop1} yields
$$
\|u^n\|_{W^{1,\infty}} \leq C \|\rho^n\|_{L^1 \cap L^\infty} \leq C.
$$
Hence, Gronwall estimate yields
$$
W_1(\rho^n,\rho^m) \leq W_1(\rho^n_0,\rho^m_0)e^{Ct}.
$$
On the other hand, if we set $X^n$ the characteristic flow associated to $u^n$ we get 
$$
\int |x| \rho^{n}(dx) = \int |X^n(t,0,x)| \rho_0^n(dx)\leq \int |x| \rho_0^n(dx) + T \underset{[0,T]}{\sup} \|u^n(t,\cdot)\|_\infty \|\rho_0^n\|_1,
$$
which ensures that $\rho^n$ has a finite first moment. Hence $(\rho^n)_n$ is a Cauchy sequence in the complete space $L^\infty(0,T;\mathcal{P}_1)$ for all $T>0$ where $\mathcal{P}_1$ is the space of finite first moment measures metrized by the Wasserstein distance $W_1$ which is complete, see \cite[Theorem 6.16]{Villani}. This ensures the existence of a limit measure $\rho \in \mathcal{P}_1$ which lies in $L^1 \cap L^\infty$ thanks to the previous bounds. On the other hand, recall that for all compact sets $K$ we have for all $m>n \geq0$
$$\|u^{n} -u^m\|_{L^\infty(|0,T],L^1(K))}+\| \nabla u^{n}-\nabla u^m\|_{L^\infty(0,T;L^1(K))} \leq  C(K)\, \| W_1(\rho^n,\rho^m)\|_{L^\infty[0,T]} .$$
Hence, $u^n_{|K}$ and $ \nabla u^n_{|K}$ are Cauchy sequences in $L^\infty(0,T;L^1(K))$ and admit a limit in $ L^\infty(0,T; W^{1,\infty}(K))$. Finally $u \in L^\infty(0,T;W^{1,\infty} \cap W^{1,1}_{\text{loc}})$. \\
Thanks to the convergence, in the space of measure-valued functions, of $\rho^n$ to $\rho$ and the strong convergence of $u^n$ towards $u$ in $L^\infty(0,T; W^{1,1}_{\text{loc}})$ one can show that $(u,\rho)$ satisfies weakly the system:
\begin{equation*}
\left\{
\begin{array}{rcll}
\partial_t \rho+{\div}(u \rho ) &=& 0,& \text{ on $[0,T] \times \mathbb{R}^3,$}\\
- \Delta u+ \nabla p &=& -\rho e_3,& \text{ on $[0,T]\times \mathbb{R}^3$},\\
\div u & =&0,& \text{ on $[0,T]\times \mathbb{R}^3$},\\
\rho(0,\cdot) &= & \rho_0, & \text{ on $\mathbb{R}^3$}.
\end{array}
\right.
\end{equation*}
Moreover, if we assume that there exists two fixed-points $(u_i,\rho_i)$, $i=1,2$, then estimate
$$
\| W_1(\rho_1,\rho_2)\|_{L^\infty[0,T]}  \leq C T\|\rho_1\| e^{C \|\rho_0\|T}\| W_1(\rho_1,\rho_2)\|_{L^\infty[0,T]} ,
$$
ensures uniqueness for $T>0$ small enough.
\end{proof}
\subsection{Proof of the Hadamard and Rybczynski result}\label{TS_sphere}
The proof of Proposition \ref{conservation_sphere} is a consequence of the following Lemma which was first proven by Hadamard and Rybczynski
\begin{lemme}[Hadamard-–Rybczynski]\label{lemme_hadamard}
Let $B_0=B(0,1)$, $u_0=-\Phi*1_{B_0}e_3$. $v^*=-\frac{4}{15}e_3$. We have
$$
(u_0-v^*)\cdot n=0 \text{ on }  \partial B_0,
$$
with $n(x)$ the unit normal vector. Sine $u_0$ is axisymmetric (invariant under any rotation with respect to the vertical axis $e_3$), it is equivalent to
\begin{eqnarray*}
(u_0(e(\theta,0))-v^* )\cdot e(\theta,0) =0,&
 \text{ for all } \theta \in [0,\pi],
\end{eqnarray*}
where $e(\theta,0)= (\sin(\theta),0,\cos(\theta) )$. 
\end{lemme}
We give below a proof relying on explicit computations
\begin{proof}
Let $\theta\in [0,\pi]$ and $e(\theta,0)\in \partial B(0,1)$. We recall formula \eqref{formule_int_u}
$$
u(e(\theta,0))=-\frac{1}{8\pi}\int_{\partial B(0,1)} \left( \frac{(e(\theta,0)-y)\cdot e_3}{|e(\theta,0)-y|}n(y)-\frac{e(\theta,0)\cdot y -1}{|e(\theta,0)-y|}e_3 \right ) d\sigma({y}).
$$
We set $Q(\theta)=\begin{pmatrix}
\cos(\theta) & 0& \sin(\theta) \\
0&1&0\\
-\sin(\theta) & 0& \cos(\theta)
\end{pmatrix}$ the rotation matrix such that $e(\theta,0)=Q(\theta) e_3$ with $e_3=(0,0,1)$ and use the change of variable $y=Q(\theta) \omega$, $\omega \in \partial B(0,1)$ such that $n(y)=y=Q(\theta) n(\omega)=Q(\theta) w$ and $d\sigma(y)=d\sigma(w)$. We drop the dependencies with respect to $\theta$ and write
\begin{align*}
-8\pi u(e)&= Q \left( \int_{\partial B(0,1)} \frac{(Q e)_3-(Q\omega)_3}{|e_3-\omega|}\omega d\sigma(\omega) \right) -\int_{\partial B(0,1)}\frac{(Qe_3)\cdot (Q\omega) -1}{|e_3-\omega|}e_3  d\sigma(\omega)\\
&=(Q e)_3 Q \left( \int_{\partial B(0,1)} \frac{w}{|e_3-\omega|}d\sigma(\omega) \right) -Q \left( \int_{\partial B(0,1)} \frac{(Q\omega)_3}{|e_3-\omega|}\omega d\sigma(\omega) \right)  \\
&-  \left(\int_{\partial B(0,1)} \frac{\omega_3}{|e_3-\omega|} d\sigma(\omega) \right) e_3 + \left(\int_{\partial B(0,1)} \frac{1}{|e_3-\omega|} d\sigma(\omega) \right) e_3.
\end{align*}
We have $(Q\omega)_3=-\sin(\theta)\omega_1+\cos(\theta)\omega_3 $, direct computations yield
\begin{eqnarray*}
\int_{\partial B(0,1)} \frac{w}{|e_3-\omega|}d\sigma(\omega)= \frac{4\pi}{3}e_3, &\displaystyle{\int_{\partial B(0,1)} \frac{1}{|e_3-\omega|}d\sigma(\omega)= 4\pi} \\ 
 \int_{\partial B(0,1)} \frac{w_1}{|e_3-\omega|}\omega d\sigma(\omega)=\frac{16}{15} \pi e_1,&\displaystyle{ \int_{\partial B(0,1)} \frac{w_3}{|e_3-\omega|}\omega d\sigma(\omega)=\frac{14}{15} 2 \pi e_3},
\end{eqnarray*}
where $e_1=(1,0,0)$ hence we get 
\begin{align*}
-8\pi u(e)&= \cos(\theta)\frac{4\pi}{3}Qe_3-Q\left(-\sin(\theta)\frac{16}{15}\pi e_1+\cos(\theta)\frac{14}{15}2\pi e_3 \right)\\
&-\frac{4\pi}{3}e_3+4 \pi e_3\\
&=-\cos(\theta) Qe_3 \frac{8\pi}{15}+\sin(\theta) \frac{16\pi}{15}Qe_1+ \frac{8\pi}{3}e_3 ,
\end{align*}
which yields the desired result.
\end{proof}

The proof of Proposition \ref{conservation_sphere} is then straightforward
\begin{proof}
Using the formulas 
$$u(t,x)=u_0(x-v^*t),\:\: \rho(t,x)=\rho_0(x-v^*t),
$$
we get for the transport-Stokes equation 
$$ \partial_t \rho+\nabla \rho\cdot u= ( \nabla \rho_0\cdot (u_0-v^*))_{|_{(\cdot-v^*t)}}=0,$$
the last equality comes from Lemma \ref{lemme_hadamard} since $\nabla \rho_0=n s^1 $ where $s^1$ is the surface measure on the sphere and $n$ the unit normal on the sphere.
\end{proof}

\section{On the derivation of a surface evolution model}
\subsection{Derivation of the hyperbolic equation}
In this part we investigate the contour evolution in the case where the initial blob is axisymmetric and can be described by a spherical parametrization
$$
B_0=\left\{r_0(\theta) \begin{pmatrix}
\cos(\phi) \sin(\theta)\\
\sin(\phi) \sin(\theta)\\
\cos(\theta)
\end{pmatrix}, (\theta,\phi)\in[0,\pi]\times[0,2\pi]\right\}.
$$
and denote by $B_t$ the domain at time $t$. In order to use a spherical parametrization we set $c(t)=(0,0,c_3(t))$ the position at time $t$ of a reference point and write $B_t= c(t)+ \tilde{B}_t$ where
$$
\tilde{B}_t=\left\{r(t,\theta) \begin{pmatrix}
\cos(\phi) \sin(\theta)\\
\sin(\phi) \sin(\theta)\\
\cos(\theta)
\end{pmatrix}, (\theta,\phi)\in[0,\pi]\times[0,2\pi]\right\}.
$$
The velocity of the point $c(t)$ can be choosen arbitrarily and in particular can be choosen such that $c(t)$ is transported along the flow meaning that $\dot{c}= u(t,c)$.\\
Using the convolution formula for the velocity field $u$ together with the weak formulation of \eqref{Vlasov_Stokes}  we get 
\begin{prop}\label{prop_model}
$r$ satisfies the following hyperbolic equation
\begin{equation*}
\left \{
\begin{array}{rcl}
\partial_t r + \partial_\theta r A_1[r] &= &A_2[r], \\
r(0,\cdot)&=& r_0.
\end{array}
\right.
\end{equation*}
In the case where the reference point $c=(0,0,c_3)$ is transported along the flow \textit{i.e.} $ u(c)=\dot{c}$ we have $c=c[r]=(0,0,c[r]_3)$ and
\begin{equation*}
\left\{
\begin{array}{rcl}
\dot{c}[r]_3(t)&=& -\frac{1}{4}\displaystyle{\int_0^\pi} r^2(t,\bar\theta) \sin(\bar \theta) \left(1-\frac{1}{2} \sin^2(\bar \theta)\right) d \bar \theta ,\\
c[r]_3(0)&=&0,
\end{array}
\right.
\end{equation*}
The operators $A_1[r]$ and $A_2[r]$ are defined as follows
\begin{multline}\label{eqA1}
A_1[r](t,\theta) := \\
-\frac{1}{8\pi r(t,\theta)} \int_0^{2\pi} \int_0^\pi \frac{r(t,\bar  \theta) \sin(\bar  \theta) - \partial_\theta r(t,\bar \theta) \cos(\bar \theta)}{\beta[r](t,\theta,\bar \theta,\phi)} r(t,\bar \theta) \sin(\bar \theta) \Big( r(t, \theta) \cos(\phi)\\-r(t,\bar \theta)\Big \{ \cos(\bar \theta) \cos( \theta)\cos(\phi)+\sin (\bar \theta) \sin( \theta)  \Big \} \Big) d \bar \theta d\phi +\frac{\dot{c}_3\sin(\theta)}{r(t,\theta)}
\end{multline}
\begin{multline}\label{eqA2}
A_2[r](t, \theta) := \\
-\frac{1}{8\pi} \int_0^{2\pi} \int_0^\pi \frac{r(t,\bar \theta) \sin(\bar \theta) - \partial_\theta r(t,\bar\theta) \cos(\bar\theta)}{\beta[r](t,\theta,\bar \theta,\phi)} r(t,\bar \theta) \sin(\bar\theta) \Big(-r(t,\bar\theta)\sin(\theta)\cos(\bar\theta)  \cos(\phi)\\+r(t,\bar\theta)\cos( \theta)\sin(\bar\theta) \Big ) d\bar \theta d\phi- \dot{c}_3 cos(\theta) \,.
\end{multline}
\begin{equation}\label{eqB}
\beta[r](\theta,\bar \theta,\phi)^2= r^2(\theta)+r^2(\bar \theta)-2 r(\theta) r(\bar \theta)(\sin (\theta) \sin (\bar \theta) \cos(\phi) + \cos(\theta) \cos(\bar \theta) ).
\end{equation}
\begin{rem}\label{consv_vol}
The volume of the drop is conserved in time
$$
 \int_0^\pi \partial_t r(t,\theta) r^2(t,\theta) \sin(\theta) d\theta = 0.
$$
\end{rem}
\end{prop}
\begin{proof}[Proof of Proposition \ref{prop_model}]
In what follows we drop the dependencies with respect to time since the operators $A_1$ and $A_2$ depend on $t$ only through $r(t,\cdot)$.\\
Using the change of variable ${x}=c(t)+\tilde{x} \in{B}_t$, $\tilde{x}\in \tilde{B}_t$ the weak formulation of the transport equation writes
$$
\int_0^T \int_{ \tilde {B}_t} \partial_t \psi + \nabla \psi \cdot ( u(c(t)+\cdot)-\dot{c})d  \tilde x =0\,, \forall \psi \in \mathcal{C}^\infty_c([0,T] \times \mathbb{R}^3),
$$
with $\dot{c}=u(c)$. Since the flow preserves the rotational invariance, we define the spherical parametrization of $ \tilde{B}_t$ as follows:
$$
 \tilde{B}_t=\left \{z e(\theta,\phi), (\theta,\phi)\in[0,\pi]\times[0,2\pi], 0\leq z \leq r(t,\theta) \right \}\,,
$$
where \begin{equation*}
e(\theta,\phi)=\left(
\begin{array}{r}
\cos(\phi)\, \sin (\theta)\\
\sin(\phi)\, \sin(\theta)\\
\cos(\theta).
\end{array}
\right).
\end{equation*}
Passing to the spherical parametrization in the weak formulation and doing an integration by parts we get for all $\psi$ compactly supported in $(0,T)\times \mathbb{R}^3$
\begin{align}\label{formule_1}
\int_0^T \int_{\tilde{B}_t} \partial_t\psi(t,\tilde{x}) d\tilde{x} &= \int_0^T \int_{[0,\pi]\times[0,2\pi]}\int_0^{r(t,\theta)}\partial_t \psi(t,ze(\theta,\phi)) z^2 \sin(\theta) dz d\theta d\phi dt \notag \\
&= - \int_0^T  \int_{[0,\pi]\times[0,2\pi]}  \psi(t,r(t,\theta)e(\theta,\phi))\partial_t r(t,\theta) r^2(t,\theta) \sin(\theta) d\theta d\phi dt  ,
\end{align}
for the second term a direct integration by parts yields
\begin{align}\label{formule_2}
&\int_0^T \int_{\tilde{B}_t} \nabla \psi(t,\tilde{x})( u(c(t)+\cdot)-\dot{c}) d\tilde{x} =\int_0^T \int_{\partial \tilde{B}_t} \psi ( u(c(t)+\cdot)-\dot{c})\cdot n d\sigma dt\notag \\
&=  \int_0^T  \int_{[0,\pi]\times[0,2\pi]}  \psi(t,r(t,\theta)e(\theta,\phi)) ( u(c(t)+r(t,\theta)e(\theta,\phi))-\dot{c})\cdot s(\theta,\phi) d\theta d\phi dt,
\end{align}
where $s$ is the surface element on $\partial \tilde{B}_t$ such that the unit normal vector satisfies $n=\frac{s}{|s|}$ and we have 
\begin{equation}\label{surface_element}
s(\theta,\phi)=s[r](\theta,\phi)= \partial_\theta \tilde{y} \times \partial_\phi \tilde{y}= r^2\, \sin(\theta)\, e(\theta,\phi)-r'(\theta)\,r(\theta)\,\sin (\theta)\, \partial_\theta e(\theta,\phi).
\end{equation}
Gathering \eqref{formule_1}, \eqref{formule_2} and \eqref{surface_element} and droping the dependencies with respect to $(t,\theta)$ we get
\begin{equation}\label{eqt}
-\partial_t r + ( u(c+re)-\dot{c})\cdot e- \frac{\partial_\theta r}{r} ( u(c+re)-\dot{c})\cdot \partial_\theta e= 0.
\end{equation}
Hence we set 
\begin{eqnarray}
\displaystyle{A_1[r]= \frac{1}{r} ( u(c+re)-\dot{c})\cdot \partial_\theta e},& A_2[r]=( u(c+re)-\dot{c})\cdot e.
\end{eqnarray}
 We recall that for all $x\in \mathbb{R}^3$:
$$
u(x)=\frac{1}{8\pi}\int_{B_t} \left(-\frac{1}{|x-y|}e_3- \frac{(x-y)\cdot e_3}{|x-y|^3} (x-y) \right),
$$
which can be reformulated using an integration by parts as follows
\begin{equation}\label{formule_int_u}
u(x)=-\frac{1}{8\pi}\int_{\partial {B}_t} \left( \frac{(x_3-{y}_3)}{|x-{y}|}n(y)-\frac{(x-{y})\cdot n(y)}{|x-{y}|}e_3 \right ) d\sigma({y}).
\end{equation}
Using again the spherical parametrization of $\partial \tilde{B}_t$, we set $y=c+\tilde{y}$, where $\tilde{y}=r(t,\bar \theta) e(\bar \theta,\bar \phi)$ and
 $$x=c+\tilde{x}=c+r(t,{\theta})\,e( {\theta}, \phi)\in \partial B_t\:\,,\:( \theta, \phi) \in [0,\pi ] \times [0,2\pi].$$ 
We recall that the velocity does not depend on the azimuth angle $ \phi$ hence we can set $ \phi=0$. We define the operator $\mathcal{U}[r]$ as $\mathcal{U}[r](t,\theta)=u(c(t)+r(t, \theta) e( \theta, 0))$  and we have
\begin{multline}\label{vitesse}
\mathcal{U}[r](t,{\theta})\\=u(c+r(t, \theta) e( \theta, 0)) = -\frac{1}{8\pi}\int_{[0,\pi]\times[0,2\pi]} \Bigg(\frac{\left (r( \theta)e( \theta,0)- r( \bar \theta)e( \bar \theta,\bar \phi)\right)\cdot e_3}{\left |r( \theta)e( \theta,0)- r(\bar  \theta)e(\bar  \theta,\bar \phi)\right|} s[r](\bar \theta,\bar \phi)\\
 - \frac{\left(r(t, \theta)e( \theta,0)- r(\bar  \theta)e(\bar  \theta,\bar \phi)\right )\cdot s[r](\bar \theta,\bar \phi)}{\left | r(\theta)e( \theta,0)- r(\bar  \theta)e(\bar  \theta,\bar \phi)\right|}e_3 \Bigg) d\bar \theta d \bar \phi .
 \end{multline}
 We recall that $A_1$ and $A_2$ are given by
 \begin{eqnarray}\label{A_1,2}
\displaystyle{A_1[r]= \frac{1}{r} (\mathcal{U}[r]-\dot{c})\cdot \partial_\theta e},& A_2[r]=( \mathcal{U}[r]-\dot{c})\cdot e.
\end{eqnarray}
We first compute the components of the vector  $\mathcal{U}[r]$. For sake of clarity we use the shortcut 
$$\beta= |x-y|=|\tilde{x}-\tilde{y}|=|r(\theta)e(\theta,0)- r( \bar \theta)e(\bar  \theta,\bar \phi)|,$$ and we have:
\begin{equation}\label{beta}
\beta^2= r^2(\theta)+r^2(\bar \theta) -2 r(\theta) r(\bar \theta) \Big(cos(\bar  \phi) \sin (\theta) \sin (\bar \theta)+ \cos (\theta) \cos(\bar \theta) \Big).
\end{equation}
This yields:
\begin{multline}\label{def_u_1}
\mathcal{U}[r]_1= -\frac{1}{8\pi} \int_0^{2\pi} \int_0^\pi \frac{r( \theta) cos( \theta) - r(\bar \theta) \cos(\bar \theta)}{\beta} r(\bar \theta) \sin(\bar \theta)\\
\times \Big(r(\bar \theta)\sin(\bar \theta)-r'(\bar \theta)\cos(\bar \theta) \Big ) \cos(\bar \phi) d\bar \theta d\bar \phi\,,
\end{multline}
\begin{multline}\label{def_u_2}
\mathcal{U}[r]_2= -\frac{1}{8\pi} \int_0^{2\pi} \int_0^\pi \frac{r( \theta) \cos( \theta) - r(\bar \theta) \cos(\bar \theta)}{\beta} r(\bar \theta) \sin(\bar \theta)\\
\times \Big(r(\bar \theta)\sin(\bar \theta)-r'(\bar \theta)\cos(\bar \theta) \Big ) \sin(\bar \phi) d\bar \theta d\bar \phi\,,
\end{multline}
\begin{multline}\label{def_u_3}
\mathcal{U}[r]_3= -\frac{1}{8\pi} \int_0^{2\pi} \int_0^\pi \frac{r( \bar \theta) \sin(\bar \theta) - r'(\bar \theta) \cos(\bar \theta)}{\beta}r(\bar \theta) \sin (\bar \theta)\\
\times \Big\{- r( \theta) \sin ( \theta) \cos(\bar \phi) +r(\bar \theta) \sin (\bar \theta) \Big\} d\bar \theta d\bar \phi
\end{multline}
We can now compute $\mathcal{U}[r]\cdot e= u \cdot e(\theta, 0)$ and $\mathcal{U}[r]\cdot \partial_\theta e= u \cdot \partial_\theta e(\theta,0)$. We get:
\begin{multline}\label{u_r}
\mathcal{U}[r]\cdot e = -\frac{1}{8\pi} \int_0^{2\pi} \int_0^\pi \frac{r(\bar  \theta) \sin(\bar  \theta) - r'(\bar \theta) \cos(\bar \theta)}{\beta} r(\bar \theta) \sin(\bar \theta) \Big(-r(\bar \theta)\sin( \theta)\cos(\bar \theta)  \cos(\bar  \phi)\\+r(\bar \theta)\cos( \theta)\sin(\bar \theta) \Big ) d\bar \theta d\bar \phi\,,
\end{multline}
\begin{multline}\label{u_teta}
\mathcal{U}[r]\cdot \partial_\theta e = -\frac{1}{8\pi} \int_0^{2\pi} \int_0^\pi \frac{r( \bar \theta) \sin(\bar  \theta) - r'(\bar \theta) \cos(\bar \theta)}{\beta} r \sin(\bar \theta) \Big( r( \theta) \cos(\bar \phi)\\-r(\bar \theta)\Big \{ \cos(\bar \theta) \cos(\theta)\cos(\bar \phi)+\sin (\bar \theta) \sin(\theta)  \Big \} \Big) d\bar \theta d\bar \phi\,.
\end{multline}
Finally if we assume $u(c)=\dot{c}$ we get:
$$
\dot{c}=u(c)=- \frac{1}{8\pi}\int_{\partial \tilde{B}_t} \left (\frac{-\tilde{y}_3}{[\tilde{y}|}s+e_3 \frac{\tilde{y} \cdot s}{|\tilde{y}|} \right) d\sigma(\tilde y)\, ,
$$
recall that $|\tilde y | =r(\theta)$ and since $e \perp\partial_\theta e$ we get:
\begin{align*} \tilde{y} \cdot s&= r(\theta)e(\theta,\phi) \cdot \left (r^2(\theta)\, \sin(\theta)\, e(\theta,\phi)-r'(\theta)\,r(\theta)\,\sin (\theta)\, \partial_\theta e(\theta,\phi) \right)\\
&= r^3(\theta) \sin(\theta).
\end{align*}
This yields:
\begin{align*}
\dot{c}_1 &= -\frac{1}{8\pi} \int \int -cos(\theta) \left ( r^2(\theta)\, \sin(\theta)\,  \cos(\phi) \sin (\theta) -r'(\theta)\,r(\theta)\,\sin (\theta)\,\cos(\phi) \cos(\theta)\right)=0\,,\\
\dot{c}_2 &= -\frac{1}{8\pi} \int \int -cos(\theta) \left ( r^2(\theta)\, \sin(\theta)\,  \sin(\phi) \sin (\theta) -r'(\theta)\,r(\theta)\,\sin (\theta)\,\sin(\phi) \cos(\theta)\right)=0\,.
\end{align*}
\begin{align*}
\dot{c}_3&=  -\frac{1}{8\pi} \int_0^{2\pi} \int_0^\pi \Big( -cos(\theta) \left ( r^2(\theta)\, \sin(\theta)\, \cos (\theta)+r'(\theta)\,r(\theta)\,\sin^2\ (\theta)\right)\notag \\
& + r^2(\theta) \sin (\theta)\Big ) d\theta d\phi\,,\notag \\
&= -\frac{1}{4} \int_0^\pi \left( r^2(\theta) \sin^3(\theta)- r'(\theta) r(\theta) \cos(\theta) \sin^2(\theta) \right ) d\theta \,,\notag \\
&=  -\frac{1}{4} \int_0^\pi \left( r^2(\theta) \sin^3(\theta)+\frac{1}{2} r^2(\theta) \Big(-\sin^3(\theta)+2\cos^2(\theta)\sin(\theta)  \Big) \right) d\theta\,,\notag \\
&=  -\frac{1}{4}  \int_0^\pi  \frac{1}{2} r^2(\theta) \Big(-\sin^3(\theta)+2\sin(\theta)  \Big)  d\theta\,,\notag \\
&=  -\frac{1}{4}\int_0^\pi   r^2(\theta)\sin(\theta) \Big(1-\frac{1}{2}\sin^2(\theta)  \Big)  d\theta  < 0.\notag \\
\end{align*}
We conclude by replacing formulas \eqref{A_1,2} and \eqref{u_r}, \eqref{u_teta} in \eqref{eqt}. For the volume conservation, direct computations using \eqref{surface_element} yield
\begin{align*}
&\int_0^\pi \partial_t r(t,\theta) r^2(t,\theta) \sin(\theta) d \theta\\
& = \int_0^\pi A_2[r](t,\theta) r^2(t,\theta) \sin(\theta) - \partial_\theta r(t,\theta) A_1[r](t,\theta) r^2(t,\theta) \sin(\theta) d\theta \\
&= \int_0^\theta r^2 \sin(\theta)(u(c+re(\theta,0))-\dot{c})\cdot  e(\theta,0)- r \partial_\theta r\sin(\theta)(u(c+re(\theta,0))-\dot{c})\cdot  \partial_\theta e(\theta,0)\\
&= \int_0^\theta (u(c+re(\theta,0))-\dot{c})\cdot s(\theta,0) d\theta \\
&= \frac{1}{2\pi}\int_0^{2\pi}\int_0^\theta u(c+re(\theta,\phi)) \cdot s(\theta, \phi) d\theta d \phi - \dot{c}_3  \int_0^\pi  \partial_{\theta}\left(\frac{1}{2}r^2(t,\theta) \sin^2(\theta) \right)d \theta \\
&= \frac{1}{2\pi} \int_{\partial \tilde{B}}u(c+x) \cdot n(x) d\sigma(x)\\
&=  \frac{1}{2\pi} \int_{\partial B} u \cdot n =0.
\end{align*}

%
\end{proof}
\subsection{Proof of the local existence and uniqueness Theorem \ref{thm2}}
This section is devoted to the proof of local existence and uniqueness of a solution for equation \eqref{equation_r_transport}. Given $r\in \mathcal{C}(0,\pi)$, we recall the definition of the following quantity 
$$
|r|_*=\underset{(0,\pi)}{\inf}r(\theta).
$$
\begin{proof}
The main idea is to apply a fixed-point argument.
We recall that the operators $A_1$ and $A_2$ are defined using the velocity field $u$ defined in \eqref{vitesse}. It is possible to formulate otherwise the velocity $u$ using a spherical parametrization of the droplet $B_t=\{c(t)+z e(\theta,\phi), (\bar \theta,\bar \phi) \in (0,\pi)\times(0,2\pi), 0 \leq z \leq r(\bar \theta) \}$. this yields the following formula for $u$ 
\begin{equation}\label{vitesse_bis}
\mathcal{U}[r]( \theta) = \int_{(0,\pi)\times(0,2\pi)}\int_0^{r(\bar \theta)} \Phi(r( \theta ) e( \theta,0)-z e(\bar \theta,\bar \phi)) z^2 \sin(\bar \theta) dz d \bar \theta d\bar \phi,
\end{equation}
with $\Phi$ the Oseen tensor, see \eqref{c1:Oseen}. With this definition, the operator $\mathcal{U}[r]$ satisfies the following estimates for $r\in W^{1,\infty}$ such that $|r|_*>1$
\begin{align*}
\left|\mathcal{U}[r]( \theta) \right|& \leq C  \int_{(0,\pi)\times(0,2\pi)}\int_0^{r(\bar \theta)}  \frac{z^2 dz}{| r( \theta ) e(\theta,0)-z e(\bar \theta,\bar \phi)|} \sin(\bar \theta) d\bar \theta d\bar \phi,\\
&\leq \frac{\|r\|_\infty^{5/2}}{\sqrt{|r|_*}}\int_{0}^\pi \frac{\sin(\bar \theta) d\bar  \theta}{|e(\bar \theta,\bar \phi)-e( \theta,0) |}
\end{align*}
where we used the fact that 
\begin{align*}
| r( \theta ) e( \theta,0)-z e(\bar \theta,\bar \phi)|^2&= z^2 +r(\theta)^2-2z r(\theta)e( \theta,0)\cdot e(\bar \theta,\bar \phi)\\
&=(z-r( \theta))^2+z r( \theta)|e(\bar \theta,\bar \phi)-e( \theta,0) |^2 \\
&\geq   z r(\theta)|e(\bar \theta,\bar \phi)-e(\theta,0) |^2,
\end{align*} 
we conclude using Lemma \ref{lemme_int_S2}. For the derivative of $\mathcal{U}[r]$ we use the shortcuts $ e=e( \theta,0$), $\bar e=e(\bar \theta,\bar \phi)$, $ r=r(\theta)$, $\bar{r}=r(\bar \theta)$ and obtain after an integration on $z$
\begin{align*}
&\left| \partial_\theta \mathcal{U}[r]( \theta) \right|\leq C (|r( \theta)|+|\partial_\theta r(\theta)|)\int_0^{2\pi}\int_{0}^\pi \int_0^{r(\bar \theta)}  \frac{z^2 dz}{| r( \theta ) e( \theta,0)-z e(\bar \theta,\bar \phi)|^2} \sin(\bar \theta) d\bar \theta d\bar \phi, \notag\\
&= C  \|r\|_{1,\infty} \int_0^{2\pi}\int_{0}^\pi\int_0^{r(\theta)}  \frac{z^2 dz}{(z-{r} {e}\cdot \bar e)^2+ {r}^2(1-\bar{e}\cdot{e}^2)} \sin(\bar \theta) d\bar \theta d\bar \phi, \notag \\
&= C  \|r\|_{1,\infty}\int_0^{2\pi}\int_{0}^\pi\Big(r(\bar \theta)+{r} e \cdot \bar{e}\log\frac{|re-\bar{r} e|}{{r}} \notag \\
&+{r} \left(\frac{2e\cdot\bar{e}^2-1}{\sqrt{1-e\cdot\bar{e}^2}}\right)\left[\arctan\left(\frac{r(\bar \theta)-{r} e\cdot\bar{e}}{{r}\sqrt{1-e\cdot\bar{e}^2}} \right) + \arctan\left(  \frac{ \bar{e}\cdot e}{\sqrt{1-e\cdot\bar{e}^2}}\right) \right] \Big) \sin(\bar \theta)d \bar \theta d \bar \phi \notag  \\
&\leq C  \|r\|_{1,\infty}^2 \Bigg(1+\frac{\|r\|_\infty}{|r|_*}\left|\log\frac{\|r\|_\infty}{|r|_*}\right| \int_0^{2\pi}\int_{0}^\pi \frac{\sin(\bar \theta) d\bar  \theta d \bar \phi}{|e(\bar \theta,\bar \phi)-e( \theta,0) |}\\
&+\int_0^{2\pi}\int_{0}^\pi \frac{\sin(\bar \theta) d\bar  \theta d \bar \phi }{\sqrt{1-e(\theta,0)\cdot e(\bar \theta, \bar \phi)^2}} \Bigg),
\end{align*}
where we used the fact that $z \log(z)$ is uniformly bounded and that $|re-\bar{r}\bar{e}|\geq |r|_*|e-\bar{e}|$. We conclude using Lemma \ref{lemme_int_S2}.\\
Let $r_1, r_2 \in \mathcal{C}(0,\pi)$, $|r_1|_*,|r_2|_*>0$, reproducing the same arguments as previously we have the following stability estimate 
\begin{align*}
&\left| \mathcal{U}[r_1](\theta) - \mathcal{U}[r_2](\theta) \right| \notag\\
 & \leq \int_{(0,\pi)\times(0,2\pi)}\left|\int_{r_1(\bar \theta)}^{r_2(\bar \theta)} \frac{z^2 dz}{|{r}_1{e}-z\bar{e}|} \right|\sin(\bar \theta)d\bar \theta d\bar  \phi \notag \\
&+\|{r}_1-{r}_2\|_\infty \int_{(0,\pi)\times(0,2\pi)}\int_0^{r_2(\bar \theta)} \left(\frac{z^2 dz}{|{r}_1{e}-z\bar e|^2}+\frac{z^2 dz}{|{r}_2{e}-z\bar e|^2} \right) dz d \bar \theta d \bar \phi \notag \\
&\leq C \|r_1-r_2\|_\infty \Bigg(\frac{\|r_1\|_\infty^{3/2}+\|r_2\|_\infty^{3/2}}{\sqrt{|r_1|_*}} +(\|r_1\|_\infty +\|r_2\|_\infty)\\
&\times \Bigg [1+(\|r_1\|_\infty +\|r_2\|_\infty)\left(\frac{1}{|r_1|_*}+ \frac{1}{|r_2|_*} \right ) \left| \log\left(\frac{(\|r_1\|_\infty +\|r_2\|_\infty)^2}{|r_1|_* |r_2|_*} \right)\right| \Bigg]\Bigg).
\end{align*}
On the other hand since $\dot{c}[r]$ is defined in \eqref{dotc_3} we get
\begin{eqnarray}\label{estimate_c}
|\dot{c}[r]| \leq C \|r\|_\infty^2,& |\dot{c}[r_1]-\dot{c}[r_2]| \leq C \| r_1-r_2\|_\infty (\|r_1\|_\infty+ \|r_2\|_\infty) .
\end{eqnarray}
Since $A_1[r]$, $A_2[r]$ are defined in \eqref{A_1,2} using the estimates of $\mathcal{U}[r]$ and $\dot{c}$ we obtain 
\begin{eqnarray}
\|A_1[r]\|_{1,\infty} &\leq &C \frac{1}{|r|_*}\left(1+\|r\|_{1,\infty}\left(1+\frac{1}{|r|_*} \right) \right)(\|\mathcal{U}[r]\|_{1,\infty}+  \|r\|_\infty^2 ),\label{estimate_1}\\
\|A_2[r]\|_{1,\infty}&\leq& C\Big( \|\mathcal{U}[r]\|_{1,\infty}+  \|r\|_\infty^2\Big), \label{estimate_2}\\
\| A_1[r_1]-A_1[r_2] \|_\infty & \leq& K\left(\frac{1}{|r_1|_*},\frac{1}{|r_2|_*}, \|r_1\|_\infty,  \|r_2\|_\infty \right) \|r_1-r_2\|_\infty \label{estimate_3}
\end{eqnarray}
Now, given $r$, we introduce  $\Theta[r]$ the characteristic flow of the transport equation \eqref{equation_r_transport}
\begin{equation*}
\left\{
\begin{array}{rcl}\dot{\Theta}[r](t,s,\theta)&=&A_1[r](t,\Theta[r](t,s,\theta)),\\
\Theta[r](t,t,\theta)&=&\theta.
\end{array}
\right.
\end{equation*}
Thanks to the regularity of $A_1[r]$ the characteristic flow is well defined and in particular the characteristic curves do not intersect and satisfy $$\Theta[r](t,s,\cdot) \circ \Theta[r](s,t,\cdot)= id .$$
In particular since $A_1[r](0)=A_1[r](\pi)=0$  we have $\Theta[r](t,s,0)=0$ and $\Theta[r](t,s,\pi)=\pi$ for all $t,s$.
Thanks to this properties, for a given $r$ the unique solution of the transport equation 
\begin{equation}\label{eqt_point_fixe}
\left\{
\begin{array}{c}
 \partial_t \tilde{r} +\partial_\theta \tilde{r} A_1[r] =A_2[r],\\
 \tilde{r}(0,\cdot)=r_0,
 \end{array}
 \right.
\end{equation}
satisfies
$$
\frac{d}{dt} \tilde{r}(t,\Theta[r](t,0,\theta))=A_2[r](t,\Theta[r](t,0,\theta)),
$$
since the characteristic curves are well defined and do not intersect we have 
\begin{equation}\label{tilde_r}
\tilde r(t,\theta)=r_0(\Theta[r] (0,t,\theta))+\int_0^t A_2[r](s,\Theta[r] (s,t,\theta))ds.
\end{equation}
Hence we define the mapping $\mathcal{L} : \mathcal{C}(0,T;\mathcal{C}^{0,1}(0,\pi)) \to  \mathcal{C}(0,T;\mathcal{C}^{0,1}(0,\pi)) $
which associates to each $r$ the solution $\tilde{r}$ of equation \ref{eqt_point_fixe} defined by \eqref{tilde_r}. Thanks to estimates \eqref{estimate_1}, \eqref{estimate_2} and \eqref{estimate_3} the operator $L$ satisfies for all $r,r_1,r_2 $ such that $\|r\|_{1,\infty}\leq \|r_0\|_{1,\infty} \lambda $ and $|r|_*\geq \beta|r_0|_*$ with $\beta <1< \lambda$
\begin{eqnarray*}
\|\mathcal{L}[r](t,\cdot)\|_{1,\infty}& \leq& \|r_0\|_{1,\infty}+ T C(\lambda,\beta, \|r_0\|_{1,\infty},|r_0|_*)\\
 |\mathcal{L}[r](t,\cdot)|_* &>&|r_0|_* - T C(\lambda,\beta,\|r_0\|_{1,\infty},|r_0|_*),\\
\|\mathcal{L}[r_1](t,\cdot)- \mathcal{L}[r_2](t,\cdot) \|_\infty &\leq&  C(\lambda,\beta,\|r_0\|_{1,\infty},|r_0|_*) T \|r_1(t,\cdot)-r_2(t,\cdot)\|_\infty.
\end{eqnarray*}
If we define the sequence $(r^n)_{n \in \mathbb{N}}$ such that $r^0=r_0$ and $r^{n+1}=L[r^n]$ \textit{i.e.}
\begin{equation}
\left\{
\begin{array}{c}
 \partial_t {r^{n+1}} +\partial_\theta r^{n+1} A_1[r^{n}] =A_2[r^n]\\
 {r}^{n+1}(0,\cdot)=r_0,
 \end{array}
 \right.
\end{equation}
 previous estimates ensure that, for $T$ small enough, $r^n$ converges (up to a subsequence) to some $\bar{r} \in \mathcal{C}(0,T;\mathcal{C}(0,\pi))$ satisfying  equation \eqref{equation_r_transport}. Moreover, we have $\bar{r}\in \mathcal{C}(0,T;\mathcal{C}^{0,1}(0,\pi))$ and $|\bar{r}|_*>0$.
Uniqueness of the fixed-point is ensured thanks to the former stability estimates. Eventually, we recover the existence and uniqueness of $c[r]$ thanks to \eqref{estimate_c}.
\end{proof}
\subsection{Analysis of the spherical case}\label{section_hyperbolic_sphere}

We are interested now in showing that if $r_0=0$, then the solution of the hyperbolic equation \eqref{equation_r_transport} corresponds also to the Hadamard-Rybczynski solution \textit{i.e.} $$c+\partial \tilde{B}_t=\partial B(v^*t,0),$$
 with  $
\partial  \tilde{B}_t=\left \{r(t,\theta) e(\theta,\phi), (\theta,\phi)\in[0,\pi]\times[0,2\pi] \right \}.
$ 
First, in the case where the reference point $c$ corresponds to the center $
c^*(t):=v^*t$
given by Hadamard-Rybczynski, the result is straightforward since the source term $A_2[r]$ of the hyperbolic equations becomes according to formula \eqref{A_1,2}
$$
A_2[r](\theta)= (\mathcal{U}[r]-\dot{c}^*)\cdot e(\theta,0)= (\mathcal{U}[r](\theta)-v^*)\cdot e(\theta,0),
$$
which vanishes for $r=1$ since $\theta \mapsto \mathcal{U}[1](\theta)$ corresponds to the velocity $\theta \mapsto u_0(e(\cdot,0))$ introduced in Lemma \eqref{lemme_hadamard}. This shows that $r=1$ is a solution to the hyperbolic equation in the case $c=c^*=v^*t$.\\
In the general case $\dot{c}\neq v^*$, by symmetry, it is enough to show that 
\begin{equation*}
|c(t)+ r(t,\theta)e(\theta,0)-c^* |^2=1 \textit{ for all }\theta\in[0,\pi] \textit{ and }t.
\end{equation*}
Equivalently, we consider the function $\bar{r}$ satisfying the above formula and show that it satisfies the hyperbolic equation. This is shown in the following Proposition.
\begin{prop}\label{prop_alkashi}
Let $r_0=1$ and $(r,c)$  the solution of \eqref{equation_r_transport} with $\dot{c}\neq v^*$. Denote by $T>0$ the maximal time of existence of the solution such that $|c-c^*| \leq 1 $ with $c^*=v^*t=-\frac{4}{15}e_3 t $. Then $r$ is given by
$$
{r}(t,\theta)=-(c-c^*)_3\cos(\theta) +\sqrt{1-(c-c^*)_3^2\sin^2(\theta)},  \: (t,\theta) \in [0,T]\times[0,\pi]
$$
and satisfies
\begin{equation*}
|c(t)+ r(t,\theta)e(\theta,0)-v^*t |^2=1 \textit{ for all }\theta\in[0,\pi] \textit{ and }t \leq T.
\end{equation*}
In other words  $$\partial B_t:=c+\partial \tilde{B}_t=\partial B(c^*,1) \text{ on }[0,T].$$

\end{prop}
\begin{proof}
First, note that $|c(t)+ {r}(t,\theta)e(\theta,0)-c^* |^2=1$ corresponds to
\begin{equation}\label{eq_deg_2}
r^2+2(c-c^*)_3r\cos(\theta)+(c-c^*)_3^2-1= 0,
\end{equation}
Computing the solutions of the quadratic equation \eqref{eq_deg_2} we denote by $\bar{r}$ the solution which satisfies $\bar{r}(0,\cdot)=1$ given by
$$
\bar{r}(t,\theta)=-(c-c^*)_3\cos(\theta) +\sqrt{1-(c-c^*)_3^2\sin^2(\theta)},
$$
which is well defined provided that $|c-c^*|\leq 1 $. We aim to prove that $\bar{r}$ satisfies the hyperbolic equation \eqref{equation_r_transport}. We have 
\begin{eqnarray*}
\partial_t \bar{r}(t,\theta)= - (\dot{c}-\dot{c}^*)\frac{\bar{r} \cos(\theta) +(c-c^*)_3}{\bar{r}+(c-c^*)\cos(\theta)} ,& \displaystyle{\partial_\theta \bar{r}(t,\theta)= \frac{\bar r(t,\theta) \sin(\theta) (c-c^*)_3}{\bar{r}+(c-c^*)_3\cos(\theta)}}.
\end{eqnarray*}
Direct computations using formula \eqref{A_1,2} yield 
\begin{multline*}
(\partial_t \bar{r} + \partial_\theta \bar{r} A_1[\bar{r}] )(\bar{r}+(c-c^*)_3\cos (\theta))= -(\dot{c}-\dot{c}^*)\bar{r} \cos(\theta) - (\dot{c}-\dot{c}^*)(c-c^*)_3\\+(\mathcal{U}[\bar{r}]-\dot{c})_1 \cos(\theta) \sin(\theta)(c_c^*)_3-(\mathcal{U}[\bar{r}]-\dot{c})_3(c-c^*)_3\sin^2(\theta).
\end{multline*}
\begin{multline*}
 A_2[\bar{r}] (\bar{r}+(c-c^*)_3\cos (\theta))=\bar{r}(\mathcal{U}[\bar{r}]-\dot{c})_1 \sin(\theta)+\\ (\mathcal{U}[\bar{r}]-\dot{c})_1\sin(\theta) \cos(\theta) (c-c^*)_3 + (\mathcal{U}[\bar{r}]-\dot{c})_3 \cos(\theta) \bar{r}+ (\mathcal{U}[\bar{r}]-\dot{c})_3 \cos(\theta)^2(c-c^*)_3.
\end{multline*}
Taking the difference between the two above formulas we obtain
$$
(\bar{r}+(c-c^*)_3\cos (\theta))(\partial_t \bar{r} + \partial_\theta \bar{r} A_1[\bar{r}] -A_2[\bar{r}])=(\dot{c}^*-\mathcal{U}[\bar{r}])\cdot(\bar{r} e(\theta,0) + c-c^*).
$$
It remains to proof that the right hand side in the above formula is equal to zero. Indeed the term  $\bar{r}+(c-c^*)_3\cos (\theta)=\sqrt{1-(c-c^*)^2\sin^2(\theta)} $ in the above left hand side cannot be identically null for all $t$ and $\theta\in[0,\pi]$ since we are in the case $\dot{c} \neq v^*$.
We recall the formula of $\mathcal{U}[{r}]$ given in \eqref{vitesse}
\begin{align*}
\mathcal{U}[{r}](t,\theta)&= -\frac{1}{8\pi}\int_{[0,\pi]\times[0,2\pi]} \Bigg(\frac{\left (r( \theta)e( \theta,0)- r( \bar \theta)e( \bar \theta,\bar \phi)\right)\cdot e_3}{\left |r( \theta)e( \theta,0)- r(\bar  \theta)e(\bar  \theta,\bar \phi)\right|} s[r](\bar \theta,\bar \phi)\\
& - \frac{\left(r(t, \theta)e( \theta,0)- r(\bar  \theta)e(\bar  \theta,\bar \phi)\right )\cdot s[r](\bar \theta,\bar \phi)}{\left | r(\theta)e( \theta,0)- r(\bar  \theta)e(\bar  \theta,\bar \phi)\right|}e_3 \Bigg)d \bar \theta d\bar \phi.
\end{align*}

We recall that $\bar{r}$ is such that $|c+\bar{r}(t,\theta)e(\theta,0)-c^*|=1 $. We claim that for all $\theta \in[0,\pi]$ there exists $\gamma \in[0,\pi]$ such that
$$c+\bar{r}(t,\theta)e(\theta,0)-c^*=e(\gamma,0),$$ 
and the mapping $\gamma \mapsto \theta$ is bijective. Indeed, let $\theta \in[0,\pi]$, we search for  $\gamma \in[0,\pi]$ satisfying
$$c+\bar{r}(t,\theta)e(\theta,0)-c^*=e(\gamma,0),$$
which yields 
\begin{eqnarray*}
\cos(\gamma)=(c-c^*_3)+\bar{r}(\theta) \cos(\theta),& \sin(\gamma)= \bar r(\theta) \sin(\theta).
\end{eqnarray*}
Note that $\displaystyle \theta \mapsto  \bar{r}(\theta) \cos(\theta) $ is monotone indeed
$$
\partial_\theta \left[\theta \mapsto \bar{r}(\theta) \cos(\theta) \right] =- \frac{r^2(\theta) \sin(\theta)}{\sqrt{1-\sin^2(\theta)(c-c^*)^2}}\leq 0,
$$
moreover,
 \begin{eqnarray*}
 \left[\theta \mapsto (c-c^*_3)+\bar{r}(\theta) \cos(\theta) \right]_{\theta=0}= 1, &\displaystyle  \left[\theta \mapsto (c-c^*_3)+\bar{r}(\theta) \cos(\theta) \right]_{\theta=\pi}= -1 
 \end{eqnarray*} 
hence we have $\theta \mapsto (c-c^*_3)+\bar{r}(\theta) \cos(\theta) \in[-1,1]$ and bijective. This ensures that  $\gamma \mapsto \theta$ is bijective and in particular we have $\gamma=0$ when $\theta=0$ and $\gamma=\pi$ when $\theta=\pi$, see Figure \ref{fig} for an illustration.\\
Consequently, we introduce the change of variable $c+\bar{r}(\bar \theta)e(\bar \theta,\bar \phi) -c^* = e(\bar \gamma,\bar \phi):=\omega \in \partial B(0,1)$ and we set $x'=c+\bar{r}(\theta)e(\theta,0) -c^*=e(\gamma,0)\in \partial B(0,1)$.
Direct computations yield
\begin{eqnarray*}
\gamma&=& \arccos((c-c^*)_3+\bar r(\theta) \cos(\theta))\\
d \gamma &=& \frac{\bar r( \theta)}{\bar r( \theta) +(c-c^*)_3\cos(\theta)} d  \theta
\end{eqnarray*}
\begin{align*}
s[\bar r](\theta,\phi)d\theta  &= \frac{\bar r( \theta) +(c-c^*)_3\cos(\theta)}{\bar r( \theta)} s[\bar r](\theta,\phi)d\gamma \\
&= \frac{\bar r( \theta) +(c-c^*)_3\cos(\theta)}{\bar r( \theta)}\bar r(\theta) \sin(\theta) (\bar r e(\theta,\phi) - \partial_\theta \bar r \, \partial_\theta e(\theta, \phi) )d\gamma\\
&= \bar r \sin(\theta)\begin{pmatrix}\bar r \sin(\theta)\cos(\phi) \\ 
\bar r \sin(\theta)\sin(\phi)\\
 \bar r \cos(\theta) +(c-c^*)\end{pmatrix} d \gamma \\
&= \sin(\gamma) e(\gamma,\phi) d \gamma \\
&= s[1](\gamma, \phi) d \gamma
\end{align*}
where we used the fact that $\sin(\gamma) = \bar r(\theta) \sin(\theta) $ and $ \cos(\gamma)= \bar r(\theta) \cos(\theta) +(c-c^*)$. We get eventually
\begin{align*}
\mathcal{U}[\bar{r}](t,\theta)&= -\frac{1}{8\pi}\int_{\partial B(0,1)} \Bigg(\frac{ (e(\gamma,0)-\omega )\cdot e_3}{\left |e(\gamma,0)-\omega\right|} n(\omega)  - \frac{(e(\gamma,0)-\omega)\cdot n(\omega)}{ |e(\gamma,0)-\omega|}e_3 \Bigg) d \sigma(\omega) \\
&= \mathcal{U}[{1}](\gamma),
\end{align*}
using lemma \ref{lemme_hadamard} and the fact that $\mathcal{U}[{1}](\cdot) $ corresponds to $u_0(e(\cdot,0))$ defined in Lemma \ref{lemme_hadamard} we have $\mathcal{U}[{1}](\gamma)\cdot e(\gamma,0)=v^*\cdot e(\gamma,0) $ which yields using the fact that $c+r(\theta)e(\theta,0) -c^*=e(\gamma,0)$
$$
\mathcal{U}[\bar{r}](t,\theta)\cdot (c+r(\theta)e(\theta,0) -c^*)=v^* \cdot (c+r(\theta)e(\theta,0) -c^*),
$$
which concludes the proof.
\end{proof}

\begin{figure}
\begin{center}
\definecolor{qqwuqq}{rgb}{0,0.39215686274509803,0}
\definecolor{qqqqff}{rgb}{0,0,1}
\begin{tikzpicture}[scale=0.8,line cap=round,line join=round,>=triangle 45,x=1cm,y=1cm]
\draw [shift={(0,1)},line width=2pt,color=qqwuqq,fill=qqwuqq,fill opacity=0.10000000149011612] (0,0) -- (27.860727762243858:0.6) arc (27.860727762243858:90:0.6) -- cycle;
\draw [shift={(0,0)},line width=2pt,color=qqwuqq,fill=qqwuqq,fill opacity=0.10000000149011612](0,0) -- (45:0.6) arc (45:90:0.6) -- cycle;
\draw [shift={(0,0)},line width=2pt]  plot[domain=-1.5707963267948966:1.5707963267948966,variable=\t]({1*3*cos(\t r)+0*3*sin(\t r)},{0*3*cos(\t r)+1*3*sin(\t r)});
\draw [line width=2pt] (0,1)-- (2.121320343559643,2.121320343559642);
\draw [line width=2pt] (0,0)-- (2.121320343559643,2.121320343559642);
\draw [line width=2pt] (0,3)-- (0,-3);
\draw (-0.6,0.44) node[anchor=north west,color=qqqqff] {$c^*$};
\draw (-0.6,1.32) node[anchor=north west, color=qqqqff] {$c$};
\draw (0.78,2.52) node[anchor=north west] {$r(\theta)$};
\draw [color=qqwuqq](0.14,2.46) node[anchor=north west] {$\theta$};
\draw [color=qqwuqq](0.17,1.) node[anchor=north west] {$\gamma$};
\draw (1.1,1.) node[anchor=north west] {1};
\draw [color=qqqqff](2.3,2.98) node[anchor=north west] {$c+r(\theta)e(\theta,0)$};
\draw [fill=qqqqff] (0,0) circle (2pt);
\draw [fill=qqqqff] (0,1) circle (2.5pt);
\draw [fill=qqqqff] (2.121320343559643,2.121320343559642) circle (2.5pt);
\end{tikzpicture}
\caption{Illustration of the bijective application  $[0,\pi]$: $\theta \mapsto \gamma$}\label{fig}
\end{center}
\end{figure}
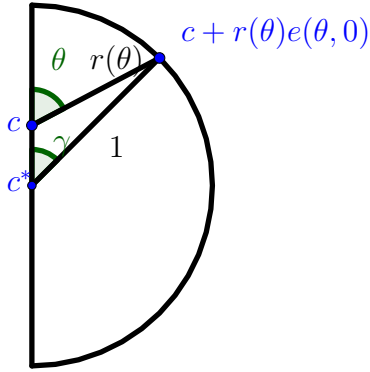

Proposition \ref{prop_alkashi} suggets that the existence time of the solution depends on the choice of $\dot{c}$. We complete the analysis by showing that the choice for which $c$ is transported along the flow \textit{i.e.} $\dot{c}$ is given by \eqref{dotc_3} is such that $|c-c^*|\leq 1$ for all time. 
\begin{prop}\label{prop_cc*}
Let $r_0=1$ and $(r,c)$  the solution of \eqref{equation_r_transport} and \eqref{dotc_3}. Then for all time $t\geq 0$ we have $c(t)\leq c^*(t)$, $|c(t)-c^*(t)|\leq 1$ and 
$$
\underset{t\to \infty}{\lim}c(t)-c^*(t)=- 1
$$

\end{prop}

\begin{proof}
We recall the formula for $r$ given by Proposition \ref{prop_alkashi}
$$
\bar{r}(t,\theta)=-(c-c^*)_3\cos(\theta) +\sqrt{1-(c-c^*)_3^2\sin^2(\theta)},
$$
and we have 
$$
r^2=1-(c-c^*)_3^2-2(c-c^*)_3r\cos(\theta).
$$
This yields 
\begin{align*}
\dot{c}_3-\dot{c}^*_3&= -\frac{1}{4} \int_0^\pi r^2(t,\theta) \sin(\theta) \left(1-\frac{1}{2}\sin^2(\theta) \right) d\theta - v^*_3\\
&=-v^*_3 -\frac{1}{4} \left(1-(c-c^*)_3^2 \right) \int_0^\pi \sin(\theta) \left(1-\frac{1}{2}\sin^2(\theta) \right) d\theta\\
& +\frac{1}{2}(c-c^*)_3\Big( -(c-c^*)_3  \int_0^\pi \cos^2(\theta) \sin(\theta) \left(1-\frac{1}{2}\sin^2(\theta) \right) d\theta \\
&+\int_0^\pi \cos (\theta) \sin(\theta) \sqrt{1-(c-c^*)_3^2\sin^2(\theta)}d\theta \Big)\\
&= -v^*_3 -\frac{1}{4} \left(1-(c-c^*)_3^2 \right) \frac{4}{3}-\frac{1}{2}(c-c^*)_3^2 \frac{8}{15}
\end{align*}
where we used the fact that the last integral vanishes using the change of variable $\theta'=\pi-\theta$. We get
$$
\dot{c}_3-\dot{c}^*_3=-\frac{1}{15}+ \frac{1}{15}(c-c^*)_3^2.
$$ 
solving the ODE $\dot{x}=-\frac{1}{15}+\frac{1}{15}x^2$ with $x(0)=0$ we obtain 
$$
{c}(t)-{c}^*_3(t)= \frac{1-e^{\frac{2t}{15}}}{e^{\frac{2t}{15}}+1},
$$

this shows that $c\leq c^*$, $|c-c^*|\leq 1$ for all time and in particular $c-c^* \to -1$ when $t\to \infty$.
\end{proof}

\subsection{Numerical simulations}
We present in this section numerical simulations in the spherical case \textit{i.e.} $r_0=1$.\\
In what follows we set $T>0$, we consider $N,M,L \in \mathbb{N}^*$  and define 
$$(\Delta t,\Delta \theta,\Delta \phi)=\left (\frac{T}{N},\frac{\pi}{M}, \frac{2 \pi}{L} \right ), $$
we set for $i=0,\cdots,M$, $j=0,\cdots, L$, $n=0,\cdots,N$
\begin{eqnarray*}
\theta_i=\Delta\theta \,i ,&\phi_j= \Delta\phi\, j,& t^n=\Delta t n.
\end{eqnarray*}
$(\theta_i)_{1 \leq i\leq M}$ is a subdivision of $[0,\pi]$, $(t^n)_{1 \leq n \leq N}$ a subdivision of $[0,T]$ and $(\phi)_{1 \leq j \leq L}$ a subdivision of $[0,2 \pi]$.
 We discretise the radius and the center by setting
$$
r(t,\theta)\sim (r_i^n)_{ 1 \leq i \leq M}^{1 \leq n \leq N}\,, r_i^n=r(t^n,\theta_i)\,, c(t) \sim (c^n)_{1 \leq n \leq N}.
$$
We use the following classical upwind finite difference scheme for the hyperbolic equation. Given $(r_i^n)_{1 \leq i \leq N}$ we define $(r_i^{n+1})_{1 \leq i \leq N}$ as 
\begin{equation}\label{scheme1}
r_i^{n+1}=r_i^n- \frac{\Delta t}{\Delta \theta} A_1^{i,n} \left\{ 
\begin{array}{rcl}
r_i^n-r_{i-1}^n &if& A_1^{i,n}\geq 0,\\
r_{i+1}^n-r_{i}^n &if& A_1^{i,n}\leq 0,
\end{array}\right.
+\Delta t A_2^{i,n}
\,, i=2,\cdots,M-1\,,
\end{equation}
where $A_1^{i,n}=A_1[r^n](t^n,\theta_i)$, $A_2^{i,n}=A_2[r^n](t^n,\theta_i)$ are computed by discretizing the integrals.
For $i=1, M$ we note that $A_1[r](t,0)=A_1[r](t,\pi)=0$ for all function $r$ and $t\geq 0$, hence we set
\begin{equation}\label{formule_bord}
r_1^{n+1}=r_1^n+\Delta t A_2^{1,n}, \:r_M^{n+1}=r_M^n+\Delta t A_2^{M,n}.
\end{equation}
For a fixed time $T>0$, the following conditions ensure a uniform bound of $\underset{1 \leq n \leq N}{\max}\underset{1 \leq i \leq M}{\max}|r_i^n| $
$$
\underset{1\leq k \leq N}{\max} \underset{1\leq i \leq M}{\max}|A_1^{i,k}| \frac{\Delta t}{\Delta \theta} <1, \: \underset{1\leq k \leq N}{\max} \underset{1\leq i \leq M}{\max}|A_2^{i,k}|\leq C .
$$
For the evolution of the center we set $c \sim (c^n)^{1 \leq n \leq N}$  with $c^0=0$.
We distinguish three test cases according the choice of the velocity of the center $c$.
\subsubsection{First test case}
The first test case corresponds to the case where $\dot{c}$ is given by \eqref{dotc_3}.
We set $(\Delta t ,M,L )=(10^{-2},100,200)$. Figure \ref{fig1} illustrates the droplet evolution on the time interval $[0,24]$ using the upwind finite difference scheme \eqref{scheme1}. Precisely we present the vertical section of the surface droplet parametrized with $\theta \mapsto (r(\theta) \sin(\theta), r(\theta) \cos(\theta))$, $\theta \in[0,\pi]$.

\begin{figure}[h]
\begin{center}
\includegraphics[scale=0.6]{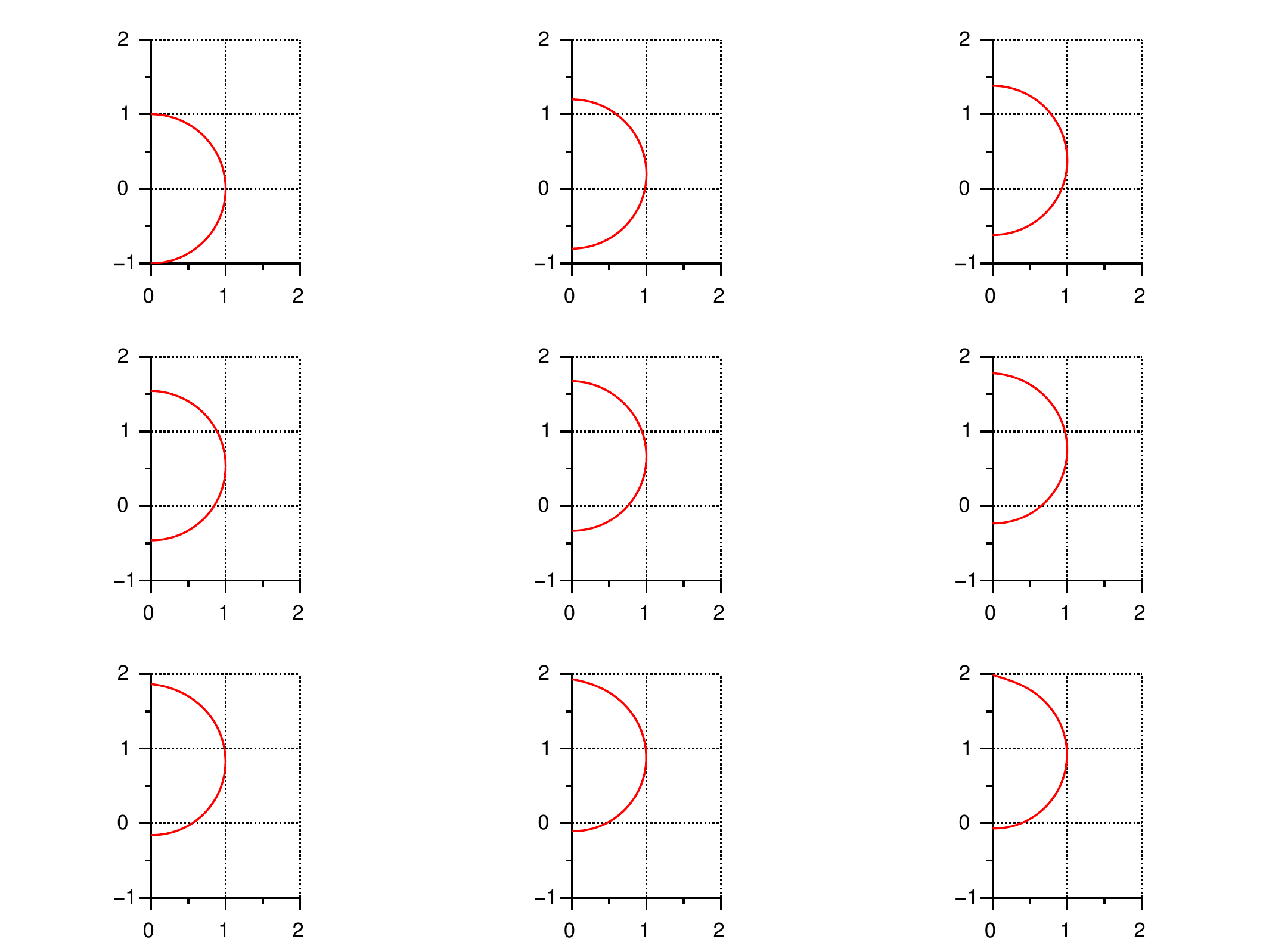}
\end{center}
\caption{First test case. Droplet evolution for $t=0,3, \cdots, 24$}\label{fig1}
\end{figure}
Table \ref{tab1} gathers the following values for each $t=0,2.5,\cdots,25$
\begin{itemize}
\item the distance $|c^n-{c^*}^n|$ between the discretized centers $c^*$ and $c$  
\item The errors $E^n $ defined by
\begin{eqnarray*}
E^n_1= \underset{i}{\max}(|r_i^n-\bar{r}(t^n,\theta_i)|),& \displaystyle E^n_2= \frac{1}{n}\underset{i}{\sum}(|r_i^n-\bar{r}(t^n,\theta_i)|), 
\end{eqnarray*}
where $\bar{r}$ is the exact solution given by Proposition \ref{prop_alkashi}
$$
\bar{r}(t,\theta)=-(c-c^*)_3\cos(\theta) +\sqrt{1-(c-c^*)_3^2\sin^2(\theta)},  \: (t,\theta) \in [0,T]\times[0,\pi].
$$
\item the relative error for the volume conservation $V^n$ defined by discretizing the integral 
\begin{equation*}
\text{Vol}(t):=\frac{2\pi}{3}\int_0^\pi r^3(\theta)\sin(\theta) d \theta= \frac{4 \pi}{3},
\end{equation*}
$$
V^n =\left|{Vol^n-\frac{4 \pi}{3}}\right|{\frac{3}{4 \pi}}.
$$
 \end{itemize}
\begin{table}[!ht]
\addtolength{\tabcolsep}{-1pt}
\begin{center}
\begin{tabular}{|l|c|c|c|c|c|c|c|c|c|c|c|}
\hline
$t$&$0$& $2.5$&$5$&$7.5$&$10$&$12.5$ &$15$ &$17.5$& $20$& $22.5$&$25$ \\
\hline
$|c-c^*|$& $7.10^{-4}$ &    $0.17$ &   $0.32$& $0.46$& $0.58$ &  $0.68$ & $0.76$ &  $0.82$ & $0.87$ &  $0.9$&$0.92$ \\
\hline 
$E_1^n (\times . 10^{-2})$ & $2. 10^{-5}$ &   $0.04$ &    $0.17$ &  $0.40$ &   $0.77$ & $1.29$& $2.02$ & $3.02$ &   $4.37$ & $6.16$&$8.55$ \\
\hline 
$ E_2^n  (\times . 10^{-2}) $& $8. 10^{-6}$ &  $0.01$&  $0.04$& $0.1$ & $0.17$ &  $0.25$ &    $0.36$ &   $0.54$ & $0.85$ & $1.33$&$2.01$ \\
\hline 
 $ V^n (\times . 10^{-3}) $ & $0.08$ & $0.15$ & $0.17$ & $0.16$ &   $0.22$ & $0.48$& $1.07$ & $2.06$ & $3.45$ &  $5.21$&$7.3$\\
 \hline
\end{tabular}
\caption{\label{tab1} First test case. Evolution of $|c-c^*|$, $E_1^n$, $E_2^n$  and  $V^n$ for the upwind finite difference scheme \eqref{scheme1}}
\end{center}
\end{table}
Numerical computations are in agreement with Proposition \ref{prop_alkashi} in the sense that $ r_i^n \sim \bar{r}(t^n,\theta_i)$ \textit{i.e.} the numerical result corresponds to the Hadamard-Rybczynski sphere which can also be noticed on Figure \ref{fig1}.

\begin{table}[!ht]
\addtolength{\tabcolsep}{-1pt}
\begin{center}
\begin{tabular}{|l|c|c|c|c|c|c|c|c|c|c|c|}
\hline
$t$&$0$& $2.5$&$5$&$7.5$&$10$&$12.5$ &$15$ &$17.5$& $20$& $22.5$&$25$ \\
\hline
$|c-c^*|$& $7.10^{-4}$ &    $0.17$ &   $0.32$& $0.46$& $0.58$ &  $0.68$ & $0.75$ &  $0.81$ & $0.84$ &  $0.87$&$0.88$ \\
\hline 
$E_1^n (\times . 10^{-2})$ & $2. 10^{-5}$ &$0.04$&$0.2$& $0.54$&   $ 1.15$&  $2.11$&  $3.5$& $5.44$& $8.04$&  $11.5$&  $ 16$  \\
\hline 
$ E_2^n (\times . 10^{-2}) $& $8. 10^{-6}$ &  $0.02$&    $0.07$&   $0.14$&  $0.27$&  $0.57$&    $1.08$& $1.86$& $2.99$& $4.51$& $6.42$ \\
\hline 
 $ V^n (\times . 10^{-3}) $ & $0.08$&  $0.3$&  $0.98$& $2.59$&   $5.16$&  $ 8.55$&  $12.5$& $16.7$&  $21.1$&  $25.4$&   $29.5$ \\
 \hline
\end{tabular}
\caption{\label{tab1_bis} First test case. Evolution of $|c-c^*|$, $E_1^n$, $E_2^n$  and  $V^n$ for the finite volume scheme \eqref{scheme2}, \eqref{flux} }
\end{center}
\end{table}

We provide in Table \ref{tab1_bis} the results obtained using the following finite volume scheme 
\begin{equation}\label{scheme2}
r_i^{n+1}=r_i^n- \frac{ \Delta t}{ \Delta \theta} \left( \mathcal{F}_{i+\frac{1}{2}}- \mathcal{F}_{i-\frac{1}{2}}\right)
+ \Delta t S^{i,n}
\,, i=1,\cdots,M-1,
\end{equation}
based on the conservative formula 
\begin{equation}\label{conservative_form}
\partial_t r + \partial_\theta (r A_1[r])= A_2[r] + r\partial_\theta A_1[r],
\end{equation}
with
\begin{equation}\label{source_n}
S^{i,n}= A_2^{i,n}+ r_i^n \frac{A_1^{i+1,n}-A_1^{i-1,n}}{2 \Delta \theta}.
\end{equation}
The flux is defined as follows 
\begin{equation}\label{flux}
\mathcal{F}_{i+\frac{1}{2}}= A_{i+\frac{1}{2}}
\left \{
\begin{array}{rcl}
r_i^n & \text{ if } & A_{i+\frac{1}{2}}\geq 0, \\
r_{i+1}^n & \text{ if } & A_{i+\frac{1}{2}}\leq0, \\
\end{array}
\right.\: A_{i+\frac{1}{2}}=\frac{A_1^{i,n}+A_1^{i+1,n}}{2}.
\end{equation}
%
Numerical computations show that $A_{i+1/2}\leq 0$ for all $i=0,\cdots,M$ and all $n=0,\cdots,N$. This means that the finite volume scheme can be rewritten as 
$$
r_i^{n+1}=r_i^n- \frac{\Delta t}{\Delta \theta}  
\frac{A_1^{i,n}+A_1^{i+1,n}}{2} (r_{i+1}^n-r_{i}^n)
+\Delta t A_2^{i,n},
$$
\begin{figure}[h!]
\begin{center}
\includegraphics[scale=0.6]{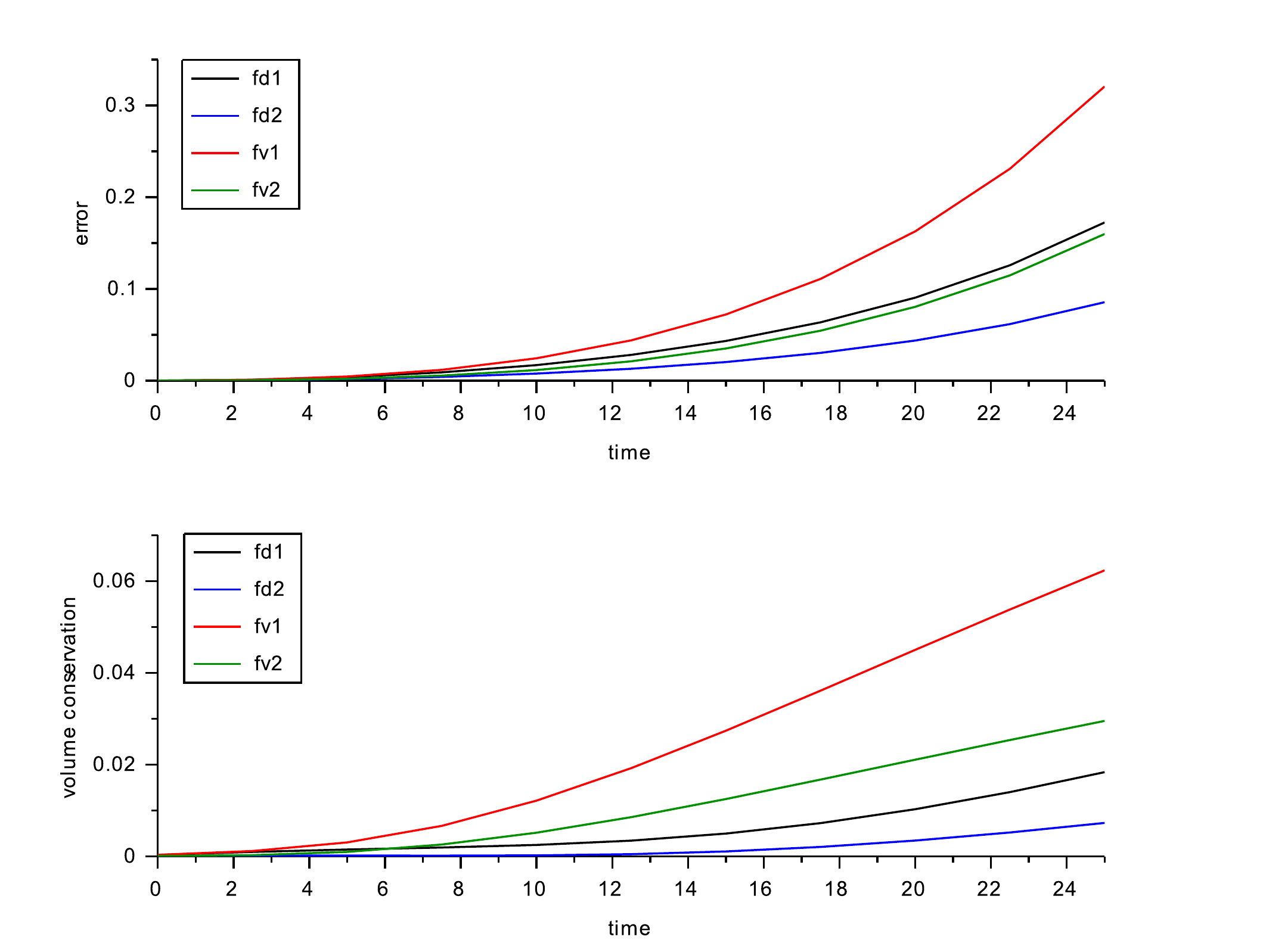}
\end{center}
\caption{Evolution of the error $E_1^n$ and the volume conservation $V^n$ for: (fd1, fd2) the finite difference scheme with ($M=50$, $M=100$) respectively, (fv1, fv2) the finite volume scheme with ($M=50$, $M=100$) respectively }\label{fig2}
\end{figure}

Figure \ref{fig2} represents the error $E_1^n$ and the volume conservation $V_n$ for the two schemes with $M= 50$ and $M=100$ on the same time interval $[0,25]$ with $(\Delta t, L)=(0.01,200)$. According to this comparison we consider only the upwind finite difference scheme for the two remaining test cases.

\subsubsection{Second test case}
The second test case is chosen such that $\dot{c}=\lambda \dot{c}^*$ with $\lambda>1$. We have
$$
|c(t)-c^*(t)|=t (\lambda-1)|v^* |= t (\lambda-1)\frac{4}{15},
$$
if we set for instance $ \lambda=\frac{17}{2}$, the time $\bar t$ for which we have $|c(\bar t)-c^*(\bar t)|=1$ is $\bar t =0.5$. We present in Table \ref{tab3} the errors computed thanks to the upwind finite difference scheme $(\Delta t, M,L)=(0.01,100,200)$.
\begin{table}[!ht]
\addtolength{\tabcolsep}{-1pt}
\begin{center}
\begin{tabular}{|l|c|c|c|c|c|c|c|c|c|c|}
\hline
$t$&$0$&   $0.1$& $0.2$&  $0.3$& $0.35$& $0.4$&  $0.45$&  $0.49$&   $0.5$& $0.51$  \\
\hline
$|c-c^*|$& $0.02$ &     $0.22$&   $0.42$&     $0.62$& $0.72$& $0.82$& $0.92$& $1.00$&   $1.02$& $1.04$\\
\hline
 $E_1^n(\times . 10^{-2})$& $ 0.02$&  $0.22$&   $0.48$&  $0.83$&  $1.08$& $ 1.4$& $1.86$& $2.51$&  $2.82$&  $64.3$  \\
\hline 
$ E_2^n(\times . 10^{-2}) $&$0.01$&  $0.11$& $0.23$&  $0.38$& $0.48$& $0.59$&  $0.69$& $0.72$& $0.92$& $2.53$ \\
\hline 
$\underset{i}{\min} \,r_i^n$&$0.98$&   $0.78$&  $0.58$&   $0.38$& $0.28$& $0.181$& $0.08$& $9.10^{-4}$& $-0.14$&$-1.34$ \\
\hline 
$V^n (\times . 10^{-2})$&  $ 0.03$&  $ 0.436$&  $ 0.87$&  $1.38$& $1.68$&  $2.02$& $2.42$&  $2.84$&  - & - \\
\hline
\end{tabular}
\caption{\label{tab3} Second test case. Evolution of  $E_1^n$, $E_2^n$, $\underset{i}{\min}\, r_i^n$  and  $V^n$}
\end{center}
\end{table}
In this case, numerical computations show that after $t=0.5$ we obtain negative values for the radius. This suggests that the maximal time of existence of the solution depends on the choice of $\dot{c}$.
\subsubsection{third test case}\label{steady}
We investigate the case where $c=c^*$ using the upwind finite difference scheme \eqref{scheme1}. In this case we recall that $\bar r=1$ is a steady solution to the hyperbolic equation. We present in Table \ref{tab2} the values of $E_1^n, E_2^n, V^n$.
\begin{table}[!ht]
\addtolength{\tabcolsep}{-1pt}
\begin{center}
\begin{tabular}{|l|c|c|c|c|c|c|c|c|c|c|c|}
\hline
$t$&$0$& $2.5$&$5$&$7.5$&$10$&$12.5$ &$15$ &$17.5$& $20$& $22.5$&$25$ \\
\hline 
$E_1^n(\times . 10^{-3})$& $4. 10^{-4}$&    $0.1$&    $0.22$&   $0.43$&   $0.78$& $1.24$&    $1.82$&   $2.53$&   $ 3.38$&    $4.41$&$5.65$\\
\hline 
$ E_2^n(\times . 10^{-3}) $& $2 . 10^{-4}$&   $ 0.06 $ &    $0.16$&   $0.3$&$0.49$&  $0.73$&    $1.02$&   $ 1.36$&    $1.76$&  $2.22$&$2.73$  \\
\hline 
 $ V^n(\times . 10^{-3}) $ & $0.08$&   $0.28$&  $0.47$&  $0.67$& $0.86$&  $1.06$&   $1.25$&  $1.44$&  $1.63$&   $1.82$&$2.02$\\
 \hline
\end{tabular}
\caption{\label{tab2} Third test case. Evolution of  $E_1^n$, $E_2^n$  and  $V^n$}
\end{center}
\end{table}

\subsubsection{Discussion on the approximation scheme}
In this last part we discuss the main difficulties encountered regarding the numerical solving of the hyperbolic equation. Several schemes have been tested in addition of the upwind finite difference scheme \eqref{scheme1} and the finite volume scheme \eqref{scheme2},\eqref{flux}. First, a Lax-Friedrichs scheme for the conservative formulation \eqref{conservative_form},   \eqref{scheme2}, \eqref{source_n} defined using the following fluxes
$$
\mathcal{F}_{i+\frac{1}{2}}= \frac{r_{i+1} A_1^{i+1,n} +r_{i} A_1^{i,n}}{2}- \frac{\Delta \theta}{2\Delta t} (r_{i+1}-r_i),
$$
yields less accurate estimate than previous schemes from the first iterations ($t\in[0,5]$) on the first test case.

 Secondly, a conservative formulation has been investigated for the hyperbolic equation which writes as follows 
$$
\partial_t r(t,\theta) + \partial_\theta G(r(t\,\theta), \theta)=A_2[r] + F(r(t,\theta),\theta),
$$
with $ \partial_r G(r,\theta)= A_1[r]$, $F(r,\theta)= \partial_\theta G(r,\theta)$. An analogous Lax-Friedrichs scheme with a discretization of the additional source term has been implemented but yields less accurate results from the first iterations ($t\in[0,2.5]$) on the first test case.

A more precise investigation of an adapted scheme for the hyperbolic equation would be interesting. In particular one of the main purposes is to ensure the steady state approximation, the positivity and the volume conservation. Keeping in mind that one of the goals is to consider different initial shapes for the droplet such as ellipsoids which correspond to the following initial conditions for instance 
\begin{eqnarray*}
 r_0(\theta)= \frac{1}{\sqrt{1-\frac{3}{4}\cos^2(\theta)}},& \displaystyle r_0(\theta)= \frac{1}{\sqrt{1-\frac{3}{4}\sin^2(\theta)}}, &\theta \in [0,\pi],
\end{eqnarray*}
depending on the considered orientation of the ellipsoid. 
\section*{Acknowledgement}
The author would like to thank Matthieu Hillairet for introducing the subject and sharing his experience. The author expresses her gratitude to Anne-Laure Dalidard for the fruitful discussions and for helping overcoming the difficulties. The author is also grateful to Jacques Sainte-Marie for helping with the numerical part. 

This project has received funding from the European Research Council (ERC) under the European Union’s Horizon 2020 research and innovation program Grant agreement No 637653, project BLOC “Mathematical Study of Boundary Layers in Oceanic Motion”. This work was supported by the SingFlows project, grant ANR-18- CE40-0027 of the French National Research Agency (ANR).        
\newpage
\appendix
\section{Summary of formulas for the operators $A_1[r]$, $A_2[r]$ and $\mathcal{U}[r]$}\label{appendiceA}

\begin{eqnarray*}
e(\theta,0)=\begin{pmatrix}\sin(\theta)\\0\\\cos(\theta))\end{pmatrix}
\end{eqnarray*}
\begin{eqnarray*}
A_1[r](\theta) ={ \frac{1}{r(\theta)} ( \mathcal{U}[r](\theta)-\dot{c})\cdot \partial_\theta e(\theta,0)}, & A_2[r](\theta)={ ( \mathcal{U}[r](\theta)-\dot{c})\cdot e(\theta,0)} 
\end{eqnarray*}
\begin{align*}
\mathcal{U}[r]_1(\theta)&= -\frac{1}{8\pi} \int_0^{2\pi} \int_0^\pi \mathcal{K}(\bar \theta, \theta, \phi) \Big \{ r( \theta) cos( \theta) - r(\bar \theta) \cos(\bar \theta) \Big\} \cos(\bar \phi) d\bar \theta d\bar \phi\\ \\
\mathcal{U}[r]_2(\theta)&= -\frac{1}{8\pi} \int_0^{2\pi} \int_0^\pi \mathcal{K}(\bar \theta, \theta, \phi) \Big\{r( \theta) cos( \theta) - r(\bar \theta) \cos(\bar \theta)\Big\} \sin(\bar \phi) d\bar \theta d\bar \phi\\ \\
\mathcal{U}[r]_3(\theta) &= -\frac{1}{8\pi} \int_0^{2\pi} \int_0^\pi \mathcal{K}(\bar \theta, \theta, \phi)\Big\{- r( \theta) \sin ( \theta) \cos(\bar \phi) +r(\bar \theta) \sin (\bar \theta) \Big\} d\bar \theta d\bar \phi \\ \\
\mathcal{K}(\bar \theta, \theta, \phi)&=  \frac{r( \bar \theta) \sin(\bar \theta) - r'(\bar \theta) \cos(\bar \theta)}{\beta[r](\bar \theta, \theta, \phi)}r(\bar \theta) \sin (\bar \theta) \\ \\
\beta^2[r](\bar \theta, \theta, \phi)&= r^2(\theta)+r^2(\bar \theta) -2 r(\theta) r(\bar \theta) \Big(cos(\bar  \phi) \sin (\theta) \sin (\bar \theta)+ \cos (\theta) \cos(\bar \theta) \Big)\\ \\
\mathcal{U}[r](\theta)\cdot e(\theta,0) &= -\frac{1}{8\pi} \int_0^{2\pi} \int_0^\pi \mathcal{K}(\bar \theta, \theta, \phi)r(\bar \theta) \Big(-\sin( \theta)\cos(\bar \theta)  \cos(\bar  \phi)\\
&+\cos( \theta)\sin(\bar \theta) \Big ) d\bar \theta d\bar \phi\\ \\
\mathcal{U}[r](\theta)\cdot \partial_\theta e(\theta,0) &= -\frac{1}{8\pi} \int_0^{2\pi} \int_0^\pi \mathcal{K}(\bar \theta, \theta, \phi) \Big( r( \theta) \cos(\bar \phi) \\
&-r(\bar \theta)\Big \{ \cos(\bar \theta) \cos(\theta)\cos(\bar \phi)+\sin (\bar \theta) \sin(\theta)  \Big \} \Big) d\bar \theta d\bar \phi
\end{align*}

\section{Proof of technical lemmas}

\begin{lemme}\label{lemme_int_S2}
There exists a positive constant $C>0$. satisfying
$$
\underset{ \theta \in[0,\pi]}{\sup} \left( \int_{[0,\pi]\times[0,2\pi]} \frac{ \sin (\bar \theta) }{|e(\bar \theta,\bar \phi)-e( \theta,0) |} d \bar \theta d \bar \phi+\int_{[0,\pi]\times[0,2\pi]} \frac{\sin(\bar \theta) d\bar  \theta d \bar \phi}{\sqrt{1-e( \theta,0)\cdot e(\bar \theta,\bar \phi)^2}} \right)  \leq C .
$$
\end{lemme}
\begin{proof}
In fact we can show a stronger result. The idea is to note that 
$$
 \int_{[0,\pi]\times[0,2\pi]} \frac{ \sin (\bar \theta) d \bar \theta d\bar \phi}{|e(\bar \theta,\bar \phi)-e( \theta,0) |} d \bar \theta d\bar \phi= \int_{\partial B(0,1)} \frac{d \sigma(y)}{|e(\theta,0) - y |}.
$$
Let $\theta \in [0,\pi]$. We set $Q(\theta)=\begin{pmatrix}
\cos(\theta) & 0& \sin(\theta) \\
0&1&0\\
-\sin(\theta) & 0& \cos(\theta)
\end{pmatrix}$ the rotation matrix such that $e(\theta,0)=Q(\theta) e_3$ with $e_3=(0,0,1)$ and use the change of variable $y=Q(\theta) \omega$, $\omega \in \partial B(0,1)$ such that  $|Q(e_3-\omega)| =|(e_3-\omega)| $ and $d\sigma(y)=d\sigma(w)$.
This yields 
$$
\int_{\partial B(0,1)} \frac{d \sigma(y)}{|e(\theta,0) - y |}=\int_{\partial B(0,1)} \frac{d \sigma(y)}{|e_3 - y |} =4\pi.
$$
We apply the same idea for the second integral using the fact that $e(\theta,0)\cdot e(\bar \theta,\bar \phi)= Q(\theta) e_3 \cdot Q(\theta) \omega = e_3 \cdot \omega $.
\end{proof}

          \end{document}